\documentclass[12pt,a4paper]{amsart}
\usepackage[latin1]{inputenc}

\usepackage{enumerate}

\usepackage{float}
\restylefloat{table}

\usepackage{graphicx}

\usepackage{hyperref}

\usepackage{bbm}

\usepackage{color}

\makeindex

\usepackage{amsmath}

\usepackage{amssymb}

\usepackage{amsthm}

\usepackage[arrow, matrix, curve]{xy}

\usepackage{color}

\newtheoremstyle{plain}
  {6pt}   
  {6pt}   
  {\itshape}  
  {0pt}       
  {\bfseries} 
  {.}         
  {5pt plus 1pt minus 1pt} 
  {}          

\newtheoremstyle{definition}
  {6pt}   
  {6pt}   
  {\normalfont}  
  {0pt}       
  {\bfseries} 
  {.}         
  {5pt plus 1pt minus 1pt} 
  {}          

\theoremstyle{plain}
\newtheorem{thm}{Theorem}[section]
\newtheorem{prop}[thm]{Proposition}
\newtheorem{cor}[thm]{Corollary}
\newtheorem{lem}[thm]{Lemma}

\theoremstyle{definition}
\newtheorem{defn}[thm]{Definition}

\newtheorem{rmk}[thm]{Remark}

\numberwithin{equation}{thm}

\newcommand{\emphbf}[1]{\emph{\textbf{#1}}}

\DeclareMathAlphabet{\mathpzc}{OT1}{pzc}{m}{it}

\DeclareMathOperator{\Hom}{Hom}
\DeclareMathOperator{\Ext}{Ext}

\DeclareMathOperator{\Mod}{Mod}
\DeclareMathOperator{\modu}{mod}

\DeclareMathOperator{\End}{End}

\DeclareMathOperator{\image}{Im}

\DeclareMathOperator{\add}{add}
\DeclareMathOperator{\Mat}{Mat}
\DeclareMathOperator{\Ind}{Ind}
\DeclareMathOperator{\CoInd}{CoInd}
\DeclareMathOperator{\twmod}{twmod}
\DeclareMathOperator{\conv}{conv}

\begin{document}

\title[QH algebras, exact Borel subalgebras, $A_\infty$-categories and boxes]{Quasi-hereditary algebras, exact Borel subalgebras, $A_\infty$-categories and boxes}
\author{Steffen Koenig, Julian K\"ulshammer and Sergiy Ovsienko}
\date{\today}

\address{Steffen Koenig, Julian K\"ulshammer\\
Institute of Algebra and Number Theory,
University of Stuttgart \\ Pfaffenwaldring 57 \\ 70569 Stuttgart,
Germany} \email{skoenig@mathematik.uni-stuttgart.de,
kuelsha@mathematik.uni-stuttgart.de}

\address{Sergiy Ovsienko\\
Mechanics and Mathematics Faculty,
Kiev National University \\
01033 Kiev, Ukraine}  
\email{ovsiyenko.sergiy@gmail.com}

\maketitle

\begin{abstract}
Highest weight categories arising in Lie theory are known to be associated with finite dimensional quasi-hereditary algebras such as Schur algebras or blocks of category $\mathcal O$. An analogue of the PBW theorem will be shown to hold for quasi-hereditary algebras: Up to Morita equivalence each such algebra has an exact Borel subalgebra. The category $\mathcal{F}(\Delta)$ of modules with standard (Verma, Weyl, \dots) filtration, which is exact, but rarely abelian, will be shown to be equivalent to the category of representations of a directed box. This box is constructed as a quotient of a dg algebra associated with the $A_{\infty}$-structure on $\Ext^*(\Delta,\Delta)$. Its underlying algebra is an exact Borel subalgebra.
\bigskip

\medskip

{\bf MR Subject Classification.}  
16G10, 16E45, 17B10, 17B35.
\end{abstract}

\tableofcontents

\section{Introduction}

Highest weight categories are abundant in algebraic Lie theory as well as in representation theory of finite dimensional algebras. A highest weight category with a finite  number of simple objects is precisely the module category of a quasi-hereditary algebra that is unique up to Morita equivalence. A general highest weight category is made up of pieces (to be reached by a well-defined truncation process) that are finite highest weight categories. Among the most frequently studied quasi-hereditary algebras are Schur algebras of reductive algebraic groups, blocks of the Bernstein-Gelfand-Gelfand category $\mathcal O$ of semisimple complex Lie algebras - or generalisations to Kac-Moody algebras - as well as path algebras of quivers, finite dimensional algebras of global dimension two and Auslander algebras.

Highest weight categories come with important objects, the standard objects, which are 'intermediate' between simple and projective
objects. Examples of standard objects include Verma and Weyl modules. The category $\mathcal{F}(\Delta)$ of objects with standard filtration is crucial in a wide variety of contexts, including Ringel's theory of (characteristic) tilting modules and Ringel duality, Kazhdan-Lusztig theory, relative Schur
equivalences and homological  questions in representation theory. In geometry, examples of standard objects come up in exceptional sequences in algebraic or symplectic geometry, where $\mathcal{F}(\Delta)$ is studied as category of 'twisted stalks'. By the Dlab-Ringel standardisation theorem (see Theorem \ref{dlabringelstandardization}) the study of the subcategory $\mathcal{F}(\Delta)$ for an exceptional sequence in any category amounts to the study of the category of filtered modules for an (up to Morita equivalence unique) quasi-hereditary algebra. For this reason we will formulate our results mostly just for quasi-hereditary algebras as this does not lose generality.

The principal goal of this article is to clarify the structure of categories $\mathcal{F}(\Delta)$ of quasi-hereditary algebras in full generality. Using $A_\infty$-techniques and differential graded categories we will establish an equivalence between $\mathcal{F}(\Delta)$ and a category of representations of a box, that is, of a representation theoretic analogue of a differential graded category, satisfying strong additional properties. This makes available for the study of $\mathcal{F}(\Delta)$ the structure theory and representation theory of boxes, which is fundamental to representation theory of algebras by Drozd's tame and wild dichotomy.

Inherent to highest weight categories and quasi-hereditary algebras is a 'directedness', which is reflected both in ordering conditions in the definition and in typical proofs in this area proceeding inductively along certain partial orders. Using the connection to boxes, this directedness can now be formulated precisely, as a characterisation of quasi-hereditary algebras and an equivalence of categories:

\begin{thm} \label{maintheorem}
A finite dimensional algebra $A$ is quasi-hereditary if and only if it is Morita equivalent to the right Burt-Butler algebra $R_{\mathfrak{B}}$ of a directed box $\mathfrak{B}$ if and only if it is Morita equivalent to the left Burt-Butler algebra $L_{\mathfrak B'}$ of a directed box
$\mathfrak{B'}$.

Moreover, the category $\mathcal{F}(\Delta)$ of $A$-modules with standard filtration is equivalent - as an exact category - to the category of representations of the box $\mathfrak B$.
\end{thm}

The if part of the theorem is a rather direct consequence of the theory set up by Burt and Butler, combined with Dlab and Ringel's 'standardisation'
technique. The converse is more involved: Starting with a quasi-hereditary algebra we consider its Yoneda algebra $\Ext^*(\Delta,\Delta)$. This carries an
$A_{\infty}$-structure, which we translate into a differential graded structure, using the concept of twisted stalks. The resulting differential graded category has additional properties telling us that this datum actually determines a box. The directedness of the underlying algebra directly comes from the directedness of homology of standard modules. The results by Burt and Butler precisely apply to the box obtained after these translations. The category of representations of the box turns out to be precisely the category $\mathcal{F}(\Delta)$.

Turning around this result it follows that every directed box $\mathfrak{B}$ produces two quasi-hereditary algebras, which in general are different. The relation between them recovers a central symmetry within the class of quasi-hereditary algebras, Ringel duality:

\begin{cor}
Let $\mathfrak B$ be a directed box and $L_{\mathfrak B}$ and $R_{\mathfrak B}$ its left and right Burt-Butler algebras. Then $L_{\mathfrak B}$ and $R_{\mathfrak B}$ are mutually Ringel dual quasi-hereditary algebras.
\end{cor}

The Main Theorem \ref{maintheorem} has a variety of applications obtained in this article, and potential for many more.

With respect to representation theory, it makes available for the study of $\mathcal{F}(\Delta)$ the rich supply of methods of representation theory of boxes. For instance we obtain as a consequence (Corollary \ref{existence-of-ASS}) a fundamental result by Ringel \cite{Rin91} asserting the existence of (relative) Auslander-Reiten sequences for boxes.

A second line of applications, and in fact the original motivation for this article, is about the relation of the algebra $B$ underlying the box $\mathfrak{B}$ and the category $\mathcal{F}(\Delta)$.

In this respect, the main application of Theorem \ref{maintheorem} is that it solves the following problem that has been open for about twenty years (after being incorrectly answered in \cite{Koe95} and \cite{Koe99}):

\begin{cor}
Every quasi-hereditary algebra is Morita equivalent to a quasi-hereditary algebra $A$ (with corresponding quasi-hereditary structure) that has an exact
Borel subalgebra $B$.
\end{cor}

Here an exact Borel subalgebra, as defined in \cite{Koe95}, satisfies properties analogous to the universal enveloping algebra of a Borel subalgebra of a semisimple complex Lie algebra, including a version of the fundamental theorem of Poincar\'e, Birkhoff and Witt.

The exact Borel subalgebra whose existence is claimed here, is the underlying algebra of the box whose representations form the category
$\mathcal{F}(\Delta)$.

Note that there are quasi-hereditary algebras without exact Borel subalgebras, see \cite{Koe95}, which implies that Theorem \ref{maintheorem} (as well as this corollary) fails to be true for isomorphism classes instead of Morita equivalence classes.

Existence of exact Borel subalgebras has been shown for the blocks of the Bernstein-Gelfand-Gelfand category $\mathcal O$ of semisimple complex Lie
algebras \cite{Koe95}, for Frobenius kernels of semisimple algebraic groups \cite{PSW00}, for various abstractly defined classes of quasi-hereditary algebras, and more recently for certain infinite dimensional algebras \cite{MM09}. Despite intensive efforts, existence for Schur algebras or for algebras of global dimension two - which, up to Morita equivalence, is a special case of our results now - could not be shown so far; more optimistic claims about Schur algebras made in \cite[appendix]{Koe95} and in \cite{Koe93} rely on incomplete proofs. 
The general existence result we are going to prove here does not close the gaps in these proofs, which are based on rather different constructions. Our results do, however, correct and replace an approach taken in \cite{Koe99}. There, a stronger concept than exact Borel subalgebras has been introduced and studied, aiming at Morita invariance. Our main result shows that the approach taken in \cite{Koe99} cannot work in general. Theorem 5.1 and Corollaries 5.2 and 5.3 in \cite{Koe99} are not correct in general; the use of the induction functors in the proofs of these results is flawed. Therefore, the consequences drawn in \cite{Koe99} and in particular the results announced in the introduction of \cite{Koe99} about applications to category $\mathcal O$ and to generalised Schur algebras are unproven. In Sections \ref{section9} and \ref{section10} we will explain how to use Burt and Butler's structure theory of boxes to set up an induction procedure that should replace the one in \cite{Koe99}. Various expectations raised by the early existence results and by the attempts to prove more such results, can now be explored again. Note that exact Borel subalgebras are different from
and satisfy stronger properties than traditional Borel Schur algebras, as introduced by Green in \cite{Gre90}. Hence, in particular further investigations in the case of algebraic groups look
promising.

The approach carried out here, to study categories $\mathcal{F}(\Delta)$ by means of $A_\infty$-structures and, in particular, boxes, has been outlined for the first time in the last named author's programmatic text \cite{Ovs05} (see also \cite{Ovs99} for related results), which in detail differs substantially from the current approach by using box techniques much more heavily, while the current approach limits the use of boxes to basic and generally accessible material and instead relies on $A_{\infty}$- and differential graded structures. These results were presented by the last named author in lecture courses in Uppsala and in K\"oln.  The article \cite{Ovs05} was in turn motivated by \cite{BKM01}, where for the first time an example of a category $\mathcal{F}(\Delta)$ has been studied using $A_\infty$-structures - or implicitly box structures - in the context of a classification of representation types.

\medskip

The structure of the article is as follows. Section \ref{section2} recalls basic notions from the representation theory of quasi-hereditary algebras, i.e. standard modules, Ringel duality and exact Borel subalgebras. In Section \ref{section3} we introduce the notion of a linear quiver, that although not necessary sometimes simplifies notation in the remainder of the article, and we fix some notation. In Section \ref{section4} we introduce the concept of an $A_\infty$-category and recall how the Yoneda algebra of a module can be regarded as an $A_\infty$-category. In Section \ref{section5} we recall the results of Keller and Lef\`evre-Hasegawa that describe the category of filtered modules $\mathcal{F}(\Delta)$ as the homology of the $A_\infty$-category of twisted modules of the $A_\infty$-category given by the Yoneda algebra. We present this $A_\infty$-category in the language of \cite{Sei08}. Section \ref{section6} translates the results of the previous section, especially the description of the objects of the homology of the $A_\infty$-category of the twisted modules to the language of representations of quivers. Section \ref{section7} introduces the notion of a box and explains how a differential graded structure may be translated into a box. In Section \ref{section8} we prove that the dg structure constructed in Section \ref{section6} can be transferred into a box. Especially the morphisms in the homology of the $A_\infty$-category of twisted modules are identified with morphisms between box representations. Section \ref{sectionexact} introduces an exact structure on the category of box representations for the boxes we have constructed in Section \ref{section7}. Section \ref{section9} then recalls the theory of boxes by Burt and Butler, e.g. introduces the structure of the left and right algebra of a box and explores their connections. The final Section \ref{section10}, then collects all the results to prove the main theorems, and gives some corollaries. Since the notions introduced in this article are quite abstract we have added an appendix which illustrates our techniques in some examples including an example of an algebra that does not have an exact Borel subalgebra, but a Morita equivalent algebra does have such a subalgebra (see \ref{A3}). In \ref{A1},
the smallest non-trivial example from algebraic Lie theory is discussed: the
principal block of category $\mathcal O$ of ${\mathfrak{sl}}(2, {\mathbb C})$.
Exact Borel subalgebras are, in general, not unique and boxes associated
with quasi-hereditary algebras aren't either: Example \ref{A2} exhibits two
boxes with the same categories of representations; one of them is connected,
while the other one is not. Example \ref{A4} provides an 
exact Borel subalgebra that is not associated with a box. More generally,
the exact Borel subalgebras constructed here are different from those 
constructed before in special situations such as blocks of category $\mathcal
O$. The boxes produced by our approach contain stronger information than arbitrary exact Borel subalgebras.

\medskip

Throughout the article we fix an algebraically closed field $\mathbbm{k}$.\index{field, algebraically closed $\mathbbm{k}$}

\section{Quasi-hereditary algebras} \label{section2}

Let $A$ be a basic finite dimensional algebra over $\mathbbm{k}$ with a fixed decomposition of the unit into a sum of primitive orthogonal idempotents $1=e_\mathtt{1}+\dots+e_\mathtt{n}$.

Unless otherwise stated our modules are left modules. By $A-\Mod$ (respectively $A-\operatorname{mod}$) we denote the category of left $A$-modules (respectively finite dimensional left $A$-modules). For a primitive idempotent $e_\mathtt{i}\in A$, we denote the corresponding simple module by $L(\mathtt{i})$, its projective cover by $P(\mathtt{i})$. We denote by $D=\Hom_\mathbbm{k}(-,\mathbbm{k}):A-\modu\to A^{op}-\modu$ the standard $\mathbbm{k}$-duality\index{standard $\mathbbm{k}$-duality $D$}.

Furthermore we let $\leq$ be the standard linear order on $\{\mathtt{1},\dots,\mathtt{n}\}$. (It is possible to work more generally with an arbitrary partial order on the simples, but by refining to a total order, all the notions will stay the same.) Let $\Delta(\mathtt{i})$ be the largest factor module of $P(\mathtt{i})$ having only composition factors $L(\mathtt{j})$ with $\mathtt{j}\leq \mathtt{i}$. The modules $\Delta(\mathtt{i})$ are called \emphbf{standard modules}.\index{standard modules $\Delta(\mathtt{i})$}

The algebra $A$ is then called \emphbf{quasi-hereditary}\index{quasi-hereditary algebra} if $\End_A(\Delta(\mathtt{i}))$ $\cong \mathbbm{k}$ for all $\mathtt{i}$ and the kernel of the natural surjection $P(\mathtt{i})\twoheadrightarrow \Delta(\mathtt{i})$ is filtered by $\Delta(\mathtt{j})$ with $\mathtt{j}>\mathtt{i}$, i.e. there exists a series of submodules of the kernel with subquotients isomorphic to $\Delta(\mathtt{j})$ with $\mathtt{j}>\mathtt{i}$.

This notion can be defined dually using submodules of the indecomposable injectives, the costandard modules $\nabla(\mathtt{i})$.\index{costandard modules $\nabla(\mathtt{i})$} In the theory of quasi-hereditary algebras, the full subcategories $\mathcal{F}(\Delta)$\index{modules filtered by standard modules $\mathcal{F}(\Delta)$} (respectively $\mathcal{F}(\nabla)$\index{modules filtered by costandard modules $\mathcal{F}(\nabla)$}) of modules which can be filtered by standard (respectively costandard) modules (i.e. there exists a series of submodules with subquotients being standard, respectively costandard modules) play an important role. Ringel introduced the notion of the characteristic tilting module $T$\index{characteristic tilting module $T=\bigoplus_{\mathtt{i}=1}^\mathtt{n}T(\mathtt{i})$}, which is the multiplicity free module such that $\mathcal{F}(\Delta)\cap \mathcal{F}(\nabla)=\operatorname{add}(T)$. Here for an $A$-module $N$, we denote by $\operatorname{add}(N)$ the full subcategory of $A-\operatorname{mod}$, which consists of all modules isomorphic to a direct summand of $N^l$ for some $l\geq 0$. Ringel proved that $T=\bigoplus_{\mathtt{i}=1}^\mathtt{n} T(\mathtt{i})$ with $T(\mathtt{i})$ indecomposable.

We denote the \emphbf{Ringel dual}\index{Ringel dual $R(A)$} of $A$ by $R(A):=\operatorname{End}_A(T)^{op}$. To indicate when we are talking about $R(A)$-modules we will add a subscript $R(A)$. The functor $F(-)=\Hom_A(T,-):A-\operatorname{mod}\to R(A)-\operatorname{mod}$ maps $T(\mathtt{i})$ to the indecomposable projective $R(A)$-module $P_{R(A)}(\mathtt{i})$. It will be convenient to use the characterisation of quasi-hereditary algebras given by the standardisation theorem by Dlab and Ringel. This also justifies only to talk about quasi-hereditary algebras, while still covering the - seemingly more general - situation of exceptional collections in general abelian $\mathbbm{k}$-categories.

\begin{thm}[{\cite[Theorem 2]{DR92}}] \label{dlabringelstandardization}
Let $\mathcal{C}$ be an abelian $\mathbbm{k}$-category. A non-empty finite set $\Delta=(\Delta_\mathtt{1},\dots,\Delta_\mathtt{n})$ of objects in $\mathcal{C}$ is called \emphbf{standardisable}\index{standardisable set of objects} (or an \emphbf{exceptional collection}) if $\End(\Delta_\mathtt{i})\cong \mathbbm{k}$, $\Hom(\Delta_\mathtt{i},\Delta_\mathtt{j})\neq 0$ implies $\mathtt{i}\leq \mathtt{j}$ and $\Ext^1(\Delta_\mathtt{i},\Delta_\mathtt{j})\neq 0$ implies $\mathtt{i}<\mathtt{j}$ and those spaces are finite dimensional. Let $\Delta$ be a standardisable set. Then there is a quasi-hereditary algebra $A$, unique up to Morita equivalence, such that the subcategory $\mathcal{F}(\Delta)$ of $\mathcal{C}$ and $\mathcal{F}(\Delta_A)$ are equivalent.
\end{thm}

We now also recall the notion of an exact Borel subalgebra introduced in \cite{Koe95}. As announced in the introduction we will prove that every quasi-hereditary algebra has such a subalgebra up to Morita equivalence.

\begin{defn}
Let $A$ be a quasi-hereditary algebra with $n$ simple modules. Then a subalgebra $B\subseteq A$ is called an \emphbf{exact Borel subalgebra}\index{exact Borel subalgebra} provided
\begin{enumerate}[{(B}1{)}]
\item The algebra $B$ has also $n$ simple modules, denoted $L_B(\mathtt{i})$ for $\mathtt{i}\in \{\mathtt{1},\dots,\mathtt{n}\}$, and $(B,\leq)$ is directed, i.e. $B$ is quasi-hereditary with simple standard modules.
\item The tensor induction functor $A\otimes_B-$ is exact.
\item There is an isomorphism $A\otimes_B L_B(\mathtt{i})\cong \Delta_A(\mathtt{i})$.
\end{enumerate}
\end{defn}

\section{Linear quivers, (co)differential tensor (co)categories}\label{section3}

\subsection{Categories over a field}

A $\mathbbm{k}$-bimodule $V$ is called \emphbf{central} if the left and right $\mathbbm{k}$-actions on $V$ coincide. A category $\mathcal{A}$ is called $\mathbbm{k}$-category or a category over $\mathbbm{k}$ if all sets $\mathcal{A}(X,Y)=\operatorname{Hom}_{\mathcal{A}}(X,Y)$ for $X,Y\in \mathcal{A}$ are central $\mathbbm{k}$-bimodules and the composition of morphisms is $\mathbbm{k}$-bilinear. Equivalently, a $\mathbbm{k}$-linear category is a category enriched over the category of $\mathbbm{k}$-vector spaces. Unless stated otherwise all categories are assumed to be $\mathbbm{k}$-categories. Functors between $\mathbbm{k}$-linear categories and natural transformations between them will also be assumed to be $\mathbbm{k}$-linear. A $\mathbbm{k}$-category $\mathcal{A}$ will be called finite dimensional provided the space $\bigoplus_{X,Y\in \mathcal{A}}\mathcal{A}(X,Y)$ is finite dimensional.

Let $S$ be a set (in many cases $S=\{\mathtt{1},\dots,\mathtt{n}\}$). The \emphbf{trivial $\mathbbm{k}$-category}\index{trivial $\mathbbm{k}$-category $\mathbb{L}$} on $S$ is defined to be the $\mathbbm{k}$-category $\mathbb{L}_S$ with objects being the elements of $S$ and only morphisms being formal multiples of the identity morphisms, i.e. $\mathbb{L}_S(\mathtt{s},\mathtt{s}')=0$ for $\mathtt{s}\neq \mathtt{s'}$ and $\mathbb{L}_S(\mathtt{s},\mathtt{s})=\mathbbm{k}\mathbbm{1}_\mathtt{s}$. We will omit the subscript if it is clear from the context.

A \emphbf{left $\mathcal{A}$-module} is  a functor $F:\mathcal{A}\to \mathbbm{k}-\operatorname{Mod}$. $F$ is called \emphbf{locally finite dimensional} if its image belongs to $\mathbbm{k}-\operatorname{mod}$. Right modules are $\mathcal{A}^{op}$-modules, where $\mathcal{A}^{op}$ is the opposite $\mathbbm{k}$-category of $\mathcal{A}$. Equivalently one can define them to be contravariant functors $\mathcal{A}\to\mathbbm{k}-\operatorname{Mod}$.

Furthermore we assume the existence of an augmentation $\pi_{\mathcal{A}}:\mathcal{A}\to \mathbb{L}_{\mathcal{A}}$ sending all morphisms apart from the scalar multiples of the identities to zero, where $\mathbb{L}_{\mathcal{A}}$ is the subcategory of $\mathcal{A}$ given by all scalar multiples of the identity morphisms. The existence of $\pi_\mathcal{A}$ is equivalent to the category algebra (or path algebra) $\mathbbm{k}[\mathcal{A}]$ of $\mathcal{A}$ to be basic; or to the assumption that $\mathcal{A}$ is skeletal and that all endomorphism algebras in $\mathcal{A}$ are local. The augmentation $\pi_\mathcal{A}:\mathcal{A}\to \mathbb{L}_{\mathcal{A}}$ defines for every $\mathtt{i}\in \mathcal{A}$ a simple module $L(\mathtt{i}): \mathcal{A}\stackrel{\pi_\mathcal{A}}{\rightarrow} \mathbb{L}_\mathcal{A}\stackrel{p_\mathtt{i}}{\rightarrow} \mathbbm{k}-\modu$, where $p_\mathtt{i}$ is given by $p_\mathtt{i}(\mathbbm{1}_\mathtt{j})=0$ for $\mathtt{j}\neq \mathtt{i}$ and $p_\mathtt{i}(\mathbbm{1}_\mathtt{i})=\mathbbm{1}_{\mathbbm{k}}$. Equivalently,  one can define the simple $\mathbb{L}_\mathcal{A}$-module $L(\mathtt{i})$ by $L(\mathtt{i}):=\mathbb{L}_\mathcal{A}(\mathtt{i},-)$ and induce it to $\mathcal{A}$ via $\pi_\mathcal{A}$. A module is called \emphbf{semisimple} if it is the direct sum of simple modules.

 If $A$ and $B$ are $\mathbbm{k}$-categories, the \emphbf{tensor category} $A\otimes B$ defined  by
$$\operatorname{Ob}(\mathcal{A}\otimes \mathcal{B}):=\operatorname{Ob}(\mathcal{A}\times \mathcal{B}); \qquad (\mathcal{A}\otimes \mathcal{B})((X_1,Y_1),(X_2,Y_2))=A(X_1,X_2)\otimes_{\mathbbm{k}}B(Y_1,Y_2)$$
is a $\mathbbm{k}$-category with composition given on simple tensors by $(f_2\otimes g_2)\circ (f_1\otimes g_1):=(f_2\circ f_1)\otimes (g_2\circ g_1)$ for $f_1\in A(X_1,X_2), f_2\in A(X_2,X_3), g_1\in B(X_1,X_2), g_2\in B(X_2,X_3)$. An \emphbf{$\mathcal{A}$-$\mathcal{B}$-bimodule} is just an $\mathcal{A}\otimes \mathcal{B}^{op}$-module.

\subsection{Linear quivers}

A $\mathbbm{k}$-linear central bimodule over $\mathbb{L}$, i.e. a $\mathbbm{k}$-linear functor $Q:\mathbb{L}\otimes \mathbb{L}^{op}\to \mathbbm{k}-\operatorname{Mod}$ is called a \emphbf{$\mathbbm{k}$-quiver}\index{linear quiver $Q$} (or \emphbf{$\mathbbm{k}$-quiver over $S$} or \emphbf{linear quiver}). Since the categories $\mathbb{L}^{op}$ and $\mathbb{L}$ are naturally isomorphic we can and will skip the $op$. Since $\mathbb{L}\otimes \mathbb{L}\cong \mathbb{L}_{S\times S}$ giving a $\mathbbm{k}$-quiver is equivalent to giving a family of vector spaces $Q(\mathtt{s},\mathtt{s}')$ for each $\mathtt{s},\mathtt{s}'\in S$. The set $S$ will be called \emphbf{objects} (or \emphbf{points}, \emphbf{vertices}) of $Q$ and will also be denoted by $Q_0$\index{vertices of a linear quiver $Q_0$ or $S$}.

Morphisms between $\mathbbm{k}$-quivers are just $\mathbbm{k}$-linear natural transformations of the underlying functors. In this way $\mathbbm{k}$-quivers form a category $\operatorname{Qui}_{\mathbbm{k}}$.

Let $\mathcal{A}$ and $\mathcal{B}$ be $\mathbbm{k}$-quivers over $\mathbb{L}$. Then we define their \emphbf{tensor product} $\mathcal{A}\otimes_{\mathbb{L}}\mathcal{B}$ as the $\mathbbm{k}$-quiver given by:
$$(\mathcal{A}\otimes_{\mathbb{L}} \mathcal{B})(\mathtt{i},\mathtt{j}):=\bigoplus_{\mathtt{k}\in \mathbb{L}}\mathcal{A}(\mathtt{k},\mathtt{j})\otimes_{\mathbbm{k}}\mathcal{B}(\mathtt{i},\mathtt{k})$$
for all $\mathtt{i},\mathtt{j}\in \mathbb{L}$.

\subsection{Graded $\mathbbm{k}$-quivers}

Let $\mathcal{C}$ be a category (regarded as a graded category concentrated in degree $0$). A \emphbf{graded module} over $\mathcal{C}$ is defined to be a family of $\mathcal{C}$-modules $\{M_i|i\in\mathbb{Z}\}$. A \emphbf{graded morphism} $f$ of degree $d$ between two graded modules $M=\{M_i\}$ and $M'=\{M'_i\}$ is defined to be a family of ($\mathbbm{k}$-bilinear) natural transformations $\{f_i:M_i\to M_{i+d}'\}$. In this way the graded $\mathcal{C}$-modules form a category. We will denote the identity morphism by $\mathbbm{1}_M:M\to M$ or just by $\mathbbm{1}$ if the module is clear from the context.

A \emphbf{graded $\mathbbm{k}$-quiver} (over $S$) is a graded $\mathbb{L}\otimes \mathbb{L}^{op}$-module. Because the only morphisms in $\mathbb{L}\otimes \mathbb{L}^{op}$ are scalar multiples of the identity morphisms a graded $\mathbbm{k}$-quiver $\mathcal{A}$ consists of a set of objects $S$ and a graded central $\mathbbm{k}$-bimodule $\mathcal{A}(\mathtt{i},\mathtt{j})$ for each pair $\mathtt{i},\mathtt{j}\in S$.

\subsection{Tensor powers and duality}

Let $\mathcal{A}$ be a graded $\mathbbm{k}$-quiver over $S$ and let $n$ be a positive integer. We will define the graded $\mathbbm{k}$-quiver $\mathcal{A}^{\otimes n}$. Here we will usually skip the subscript $\mathbb{L}$. For all $\mathtt{i},\mathtt{j}\in \mathcal{A}$ define the graded $\mathbbm{k}$-bimodule $\mathcal{A}^{\otimes n}(\mathtt{i},\mathtt{j})$ by the following formula:
$$(\mathcal{A}^{\otimes n}(\mathtt{i},\mathtt{j}))_i=\bigoplus_{\substack{(i_1,\dots,i_n)\in \mathbb{Z}^n\\ i=i_1+\dots+i_n}}\bigoplus_{\mathtt{i}_1,\dots,\mathtt{i}_{n-1}\in \mathcal{A}}\mathcal{A}(\mathtt{i}_{n-1},\mathtt{j})_{i_n}\otimes \mathcal{A}(\mathtt{i}_{n-2},\mathtt{i}_{n-1})_{i_{n-1}}\otimes \dots\otimes \mathcal{A}(\mathtt{i},\mathtt{i}_1)_{i_1},$$
For $n=0$ we set $\mathcal{A}^{\otimes 0}:=\mathbb{L}$.

Note that by abuse of notation we often write $(a_1,\dots,a_n)$ instead of $a_1\otimes \dots \otimes a_n\in \mathcal{A}^{\otimes n}$ to make things more readable when several tensor products at different 'levels' occur.

\begin{rmk}\label{remark1}
If $\mathcal{A}$ and $\mathcal{B}$ are linear quivers and $n\geq 0$, then because of the universal property of the direct sum giving a $\mathbbm{k}$-quiver morphism $F:\mathcal{A}^{\otimes n}\to \mathcal{B}$ is the same as giving the restriction of $F$ on each of the direct summands $\mathcal{A}(\mathtt{i}_{n-1},\mathtt{j})\otimes \mathcal{A}(\mathtt{i}_{n-2},\mathtt{i}_{n-1})\otimes \dots\otimes \mathcal{A}(\mathtt{i},\mathtt{i}_1)$.
\end{rmk}

If $\mathcal{A}$ and $\mathcal{B}$ are graded $\mathbbm{k}$-quivers over $S$, then $\Hom_{\mathbb{L}}(\mathcal{A},\mathcal{B})$ is a graded vector space via:
$$\Hom_{\mathbb{L}}(\mathcal{A},\mathcal{B})_k=\prod_{i\in \mathbb{Z}}\Hom_\mathbb{L}(\mathcal{A}_i,\mathcal{B}_{i+k}).$$
Let $\mathcal{A}$ be a category and $N\in \mathbb{N}$. Consider graded $\mathcal{A}$-modules $X_i, Y_i$ and graded morphisms $f_i:X_i\to Y_i$ for $i=1,\dots, N$. We will use the \emphbf{Koszul-Quillen sign rule}\index{Koszul-Quillen sign rule} for the action of the tensor product $f_N\otimes \dots\otimes f_1:X_N\otimes \dots\otimes X_1\to Y_N\otimes \dots\otimes Y_1$, i.e. for elements $x_i\in X_i$ the following equality holds
$$(f_N\otimes \dots\otimes f_1)(x_N\otimes \dots\otimes x_1)=(-1)^{\varepsilon(\overleftarrow{f},\overleftarrow{x})}f_N(x_N)\otimes\dots\otimes f_1(x_1)$$
where
$$\varepsilon(\overleftarrow{f},\overleftarrow{x}):=\sum_{1\leq i<j\leq N}|f_i|\cdot |x_j|,$$
\index{switching sign $\varepsilon(\overleftarrow{v},\overleftarrow{w})$}with $\overleftarrow{f}=(f_N,\dots,f_1)$ and $\overleftarrow{x}=(x_N,\dots,x_1)$, i.e. the sign indicates how often $f_i$ and $x_j$ are switched before arriving at the correct position.\\
If $M=\{M^i\}$ is a graded $\mathbbm{k}$-module we denote by $sM$ the graded $\mathbbm{k}$-module given by $(sM)^i=M^{i+1}$ for all $i\in \mathbb{Z}$. We call $sM$ the \emphbf{suspension}\index{suspension $s$} (or the \emphbf{shift} of $M$). To distinguish elements of $M$ from those of  $sM$ we will denote the image of an element $x\in M$ in $sM$ by $sx$, hence $|sx|=|x|-1$.

There is a duality $\mathbb{D}$\index{duality $\mathbb{D}$} on the category of graded $\mathbbm{k}$-quivers given by $(\mathbb{D}\mathcal{A})(\mathtt{i},\mathtt{j})_k=D(\mathcal{A}(\mathtt{i},\mathtt{j})_{-k})$ for all $\mathtt{i},\mathtt{j}\in \mathcal{A}$, where as before $D=\Hom_{\mathbbm{k}}(-,\mathbbm{k})$ denotes the usual $\mathbbm{k}$-duality on vector spaces. If $\mathbb{L}$ is finite, and $\mathcal{A}(\mathtt{i},\mathtt{j})$ and $\mathcal{B}(\mathtt{i},\mathtt{j})$ are finite dimensional for all $\mathtt{i},\mathtt{j}\in \mathcal{A}$, then
$$\mathbb{D}(\mathcal{A}\otimes \mathcal{B})\cong (\mathbb{D}\mathcal{A})\otimes(\mathbb{D}\mathcal{B}).$$

\subsection{Tensor (pre)category}

Define the (reduced) tensor graded $\mathbbm{k}$-quiver $\overline{T}\mathcal{A}$\index{tensor quiver/algebra} of a graded quiver $\mathcal{A}$ as follows: $\operatorname{Ob}\overline{T}\mathcal{A}=\operatorname{Ob}\mathcal{A}$ and the graded $\mathbbm{k}$-module of morphisms is defined by
$$\overline{T}\mathcal{A}(\mathtt{i},\mathtt{j})=\bigoplus_{n>0}\mathcal{A}^{\otimes n}(\mathtt{i},\mathtt{j})$$
for all $\mathtt{i},\mathtt{j}\in \mathcal{A}$. The $\mathbbm{k}$-quiver $\overline{T}\mathcal{A}$ can be endowed with either the structure of a precategory (also called semicategory) or of a precocategory\footnote{Precategory (respectively precocategory) means that this structure may not have identity morphisms (respectively counits).}. The structure of a precategory is given by $m:\overline{T}\mathcal{A}\otimes \overline{T}\mathcal{A}\to \overline{T}\mathcal{A}$ in the usual way by 'concatenation':
$$(a_n\otimes \dots\otimes a_1)\cdot (a_m'\otimes\dots\otimes a_1'):=(a_n\otimes \dots\otimes a_1\otimes a_m'\otimes \dots \otimes a_1').$$
The restriction of the comultiplication $\Delta^{co}:\overline{T}\mathcal{A}\to \overline{T}\mathcal{A}\otimes \overline{T}\mathcal{A}$ on $\overline{T}\mathcal{A}(\mathtt{i},\mathtt{j})\to \bigoplus_{\mathtt{k}\in \mathcal{A}}\overline{T}\mathcal{A}(\mathtt{k},\mathtt{j})\otimes  \overline{T}\mathcal{A}(\mathtt{i},\mathtt{k})$ is given by
$$\Delta^{co}(a_n\otimes \dots\otimes a_1)=\sum_{i=1}^{n-1}(a_n\otimes \dots\otimes a_{i+1})\otimes (a_{i}\otimes \dots\otimes a_1).$$

In the case of multiplication we will also consider the \emphbf{tensor category} $T\mathcal{A}=\bigoplus_{i=0}^\infty\mathcal{A}^{\otimes n}$, where $\mathcal{A}^{\otimes 0}=\mathbb{L}$ and the multiplication $m$ is defined as above with $m(\mathbbm{1}_{\mathtt{j}},f)=f=m(f,\mathbbm{1}_{\mathtt{i}})$ and $m(\mathbbm{1}_{\mathtt{k}},f)=0=m(f,\mathbbm{1}_{\mathtt{k}})$ for $f\in \mathcal{A}(\mathtt{i},\mathtt{j})$ and $\mathtt{k}\neq \mathtt{i},\mathtt{j}$. We will also denote the tensor category by $\mathbb{L}[\mathcal{A}]$.

A $\mathbbm{k}$-quiver map $d:T\mathcal{A}\to T\mathcal{A}$  of degree $k$ is called a \emphbf{derivation} if $d\circ \mu=\mu\circ (\mathbbm{1}\otimes d+d\otimes \mathbbm{1})$.

A $\mathbbm{k}$-quiver morphism of degree zero $f:\overline{T}\mathcal{A}\to \overline{T}\mathcal{B}$ is called a \emphbf{cofunctor} provided that $\Delta^{co} f=(f\otimes f)\Delta^{co}$. A mapping of graded $\mathbbm{k}$-quivers $b:\overline{T}\mathcal{A}\to \overline{T}\mathcal{A}$ of degree $k$ is called a \emphbf{coderivation} of degree $k$ provided that $\Delta^{co} b=(b\otimes \mathbbm{1}+ \mathbbm{1}\otimes b)\circ \Delta^{co}$.

\section{$A_\infty$-categories}\label{section4}

\subsection{Definitions of $A_\infty$-categories}

There are many equivalent definitions of an $A_\infty$-category, which differ mostly in signs and the direction of morphisms. We restrict ourselves to the definition most useful for our purposes. This is the definition which is based on the bar-construction together with the Koszul-Quillen sign convention as in \cite[3.6]{K01}. Usually an $A_\infty$-category is defined using chains of objects. Using Remark \ref{remark1} we use the equivalent language of $\mathbb{L}$-modules or graded quivers (for an introduction to the subject see also \cite{L-H03, BLM08, LM08}).

\begin{defn}\index{A-infinity category@$A_\infty$ category}
A graded $\mathbbm{k}$-quiver $\mathcal{A}$ is called a (non-unital) \emphbf{$A_\infty$-category} provided there is a graded coderivation $b:\overline{T}s\mathcal{A}\to \overline{T}s\mathcal{A}$ of degree $1$ such that $b^2=0$.
\end{defn}

Due to the following lemma the condition $b^2=0$ can be rewritten as the following family of equalities $(A_n)$, where $n\in \mathbb{N}$, which gives another definition of $A_\infty$-category:

\begin{lem}[{\cite[Lemma 3.6]{K01}}]
To define a (non-unital) $A_\infty$-category is equivalent to defining linear quiver mappings $b_n:(s\mathcal{A})^{\otimes n}\to s\mathcal{A}$ of degree $1$ such that the following relations hold for every $n\in \mathbb{N}$:
$$A_n: \sum_{j=0}^{n-1}\sum_{k=1}^{n-j}b_{n-k+1}(\mathbbm{1}^{\otimes (n-j-k)}\otimes b_k\otimes \mathbbm{1}^{\otimes j})=0.$$
\end{lem}

An \emphbf{$A_\infty$-functor} between two $A_\infty$-categories $f:\mathcal{A}\to \mathcal{B}$ is a cofunctor $\overline{T}f:\overline{T}s\mathcal{A}\to \overline{T}s\mathcal{B}$ which commutes with the differentials, i.e.
$$(\overline{T}f\otimes \overline{T}f)\circ \Delta^{co}=\Delta^{co}\circ \overline{T}f,\quad b\circ \overline{T}f=\overline{T}f\circ b.$$

\subsection{Strictly unital and augmented $A_\infty$-algebras}

An $A_\infty$-category $\mathcal{A}$ is called \emphbf{strictly unital}\index{A-infinity category@$A_\infty$ category!strictly unital} provided that for every $\mathtt{i}\in \mathcal{A}$ there exists an element $s\mathbbm{1}_{\mathtt{i}}\in s\mathcal{A}^{-1}(\mathtt{i},\mathtt{i})$, such that for every $sx\in s\mathcal{A}(\mathtt{i},\mathtt{j})$ holds:
$$b_2(s\mathbbm{1}_\mathtt{j}\otimes sx)=sx, b_2(sx\otimes s\mathbbm{1}_\mathtt{i})=(-1)^{|x|}sx, \text{ and } b_n(\cdots\otimes s\mathbbm{1}_\mathtt{i}\otimes \cdots)=0 \text{ for } n\neq 2. $$
This may seem a bit strange but is the translation (via the bar construction) of the fact that $\mathbbm{1}_\mathtt{i}$ behaves as a unit with respect to the multiplication $m_2$ (when removing the $s$).

Obviously, $\mathbb{L}$ has the structure of a strictly unital $A_\infty$-category over itself if we define $\mathbb{L}^0=\mathbb{L}$, $\mathbb{L}^i=0$ for $i\neq 0$ with a unique nonzero multiplication $b_2$. For a strictly unital $A_\infty$-category $\mathcal{A}$ the embedding $\iota:\mathbb{L}\hookrightarrow \mathcal{A}$ is an $A_\infty$-functor. The category $\mathcal{A}$ is called \emphbf{augmented}, provided there exists an augmentation, i.e. an $A_\infty$-functor $\eta:\mathcal{A}\to \mathbb{L}$ such that $\eta \iota=\operatorname{id}_{\mathbb{L}}$.

\subsection{The Yoneda algebra as an $A_\infty$-category and Kadeishvili's theorem}

For a quasi-hereditary algebra $A$ the $A_\infty$-category we have in mind is the Yoneda algebra of the set of standard modules $\Ext^*_A(\oplus_{\mathtt{i}}^{\mathtt{n}} \Delta(\mathtt{i}),\oplus_{\mathtt{i}=\mathtt{1}}^{\mathtt{n}}\Delta(\mathtt{i}))$. This can be constructed as follows. First, one takes a projective resolution $\mathbf{P}(\mathtt{i})$ of $\Delta(\mathtt{i})$ for each $\mathtt{i}$. Denote by $\mathbf{P}^\oplus(\mathtt{i})$ the direct sum of the projective modules occuring in $\mathbf{P}(\mathtt{i})$. Then $\oplus_{j\in \mathbb{Z}}\Hom_A(\oplus_{\mathtt{i}=\mathtt{1}}^{\mathtt{n}}\mathbf{P}^{\oplus}(\mathtt{i}),\oplus_{\mathtt{i}=\mathtt{1}}^{\mathtt{n}}\mathbf{P}^{\oplus}(\mathtt{i})[j])$ has a natural structure of a dg category (i.e. an $A_\infty$-category with $b_n=0$ for $n\geq 3$) where the underlying graded $\mathbbm{k}$-quiver is given by $(\mathtt{i},\mathtt{j})\mapsto \oplus_{j\in \mathbb{Z}}\Hom_A(\mathbf{P}^{\oplus}(\mathtt{i}), \mathbf{P}^{\oplus}(\mathtt{j})[j])$, its $j$-th component being the group of homogeneous $A$-linear maps $f:\mathbf{P}^{\oplus}(\mathtt{i})\to \mathbf{P}^{\oplus}(\mathtt{j})$ of degree $j$ (no compatibility with $d_\mathbf{P}$), $b_1$ induced by the differential of the projective resolution $d_\mathbf{P}$ via $b_1(f)=d_{\mathbf{P}}\circ f-(-1)^jf\circ d_\mathbf{P}$ for $f$ homogeneous of degree $j$ and $b_2$ being the natural composition of graded maps. (compare e.g. \cite[(3.3)]{K01}) We will denote this $A_\infty$-category as $\Hom_A(\mathbf{P},\mathbf{P})$.

In a next step the following theorem by Kadeishvili et alii can be invoked. Before stating it, let us shortly recall the definition of the homology category. For an $A_\infty$-category $\mathcal{A}$ its homology category $H^*\mathcal{A}$ is the $\mathbbm{k}$-quiver $Z^*\mathcal{A}/B^*\mathcal{A}$, where the objects $Z^*\mathcal{A}$ are those of $\mathcal{A}$ and its morphisms are those morphisms $f$ of $\mathcal{A}$ with $b_1(f)=0$. Similarly the objects of $B^*\mathcal{A}$ are those of $\mathcal{A}$ and its morphisms are all $b_1(f)$ for morphisms $f$ of $\mathcal{A}$. In both cases the grading is induced by that of $\mathcal{A}$. We denote the corresponding degree zero parts by $H^0\mathcal{A}$, $Z^0\mathcal{A}$ and $B^0\mathcal{A}$, respectively.  

\begin{thm}[{\cite{Kad82}, see also \cite{Kad80, Smi80, Pro84, GLS91, JL01, Mer99}}]\label{Kadeishvili}
Let $\mathcal{A}$ be an $A_\infty$-category. Then the homology category $H^*\mathcal{A}$ (where the homology is taken with respect to $b_1$) carries the structure of an $A_\infty$-category with $b_1=0$ and $b_2$ induced by the $b_2$ of $\mathcal{A}$.
\end{thm}

Since the quiver given by $(\mathtt{i},\mathtt{j})\mapsto \Ext^*(\Delta(\mathtt{i}),\Delta(\mathtt{j}))$ is the homology of $\Hom_A(\mathbf{P},\mathbf{P})$ there is an $A_\infty$-structure on this quiver. We will denote this $A_\infty$-category as $\Ext^*(\Delta,\Delta)$. The $b_i$ on $s\Ext^*(\Delta,\Delta)$ can be constructed inductively, see \cite[(7.8, 7.9)]{K01} and \cite[Appendix B]{Mad02} for an explanation and examples.

\section{Filtered modules as twisted stalks}\label{section5}

In \cite[(7.4)]{K01} and \cite[(2)]{K02} categories of the form $\mathcal{F}(M)$, where $M=(M_1,\dots,M_r)$ are modules over an associative algebra, are rewritten in terms of twisted objects over an $A_\infty$-category $\mathcal{A}$. Here we will recall these constructions modifying them in the spirit of \cite{Sei08}. Furthermore, since in this paper we are only interested in the subcategory (of the category of $A$-modules) of filtered modules, and not in the triangulated hull of the standard modules in the derived category of $A$, we adapt the notation of \cite{Mad02} and speak of twisted modules for the $A_\infty$-category $\mathcal{A}$. In this notation the $\mathbb{Z}$-graded (or derived) version of it are the twisted complexes. Throughout this section let $\mathcal{A}$ be an $A_\infty$-category.

\subsection{The category $\add \mathcal{A}$}

In a first step we construct what we call the \emphbf{additive enlargement}\index{additive enlargement $\add\mathcal{A}$} $\add \mathcal{A}$ of $\mathcal{A}$. Its $\mathbb{Z}$-graded version is denoted $\Sigma \mathcal{A}$ in \cite{Sei08}, or in a slightly different form (that can be obtained by  choosing bases of the vector spaces appearing) $\Mat\mathbb{Z}\mathcal{A}$ in \cite{K01}. As usual we present it using the bar-construction:
\begin{itemize}
\item The objects of $\add \mathcal{A}$ are the $\mathbb{L}$-modules.
\item For $X,Y\in \add \mathcal{A}$ we define the graded vector space $(\add \mathcal{A})(X,Y)$ by
\[\add \mathcal{A}(X,Y)_k=\bigoplus_{\mathtt{i},\mathtt{j}=\mathtt{1}}^{\mathtt{n}}\Hom_{\mathbbm{k}}(X(\mathtt{i}),Y(\mathtt{j}))\otimes_{\mathbbm{k}} \mathcal{A}(\mathtt{i},\mathtt{j})_k.\]
This can also be regarded as a graded $\mathbb{L}\otimes \mathbb{L}$-module via $(\mathtt{i},\mathtt{j})\mapsto \Hom_{\mathbbm{k}}(X(\mathtt{i}),Y(\mathtt{j}))\otimes \mathcal{A}(\mathtt{i},\mathtt{j})$ for $\mathtt{i},\mathtt{j}\in \mathcal{A}$, which amounts to an additional 'indexing'. Note that $\Hom_\mathbbm{k}(X(\mathtt{i}),Y(\mathtt{j}))$ is assumed to have degree $0$. All the grading is in $\mathcal{A}$ and hence also the shift only acts on the second tensor factor.
\item The graded multiplications $b_n^a$ are given by:
$$b_n^a(f_n\otimes sa_n,\dots,f_1\otimes sa_1)=-f_n\circ \cdots\circ f_1\otimes b_n(sa_n\otimes\dots\otimes  sa_1).$$
\end{itemize}

\begin{lem}
This defines on $\add\mathcal{A}$ the structure of an $A_\infty$-category.
\end{lem}

\begin{proof}
First recall that $$b_k^a(f_{j+k}\otimes sa_{j+k},\dots,f_{j+1}\otimes sa_{j+1})=-f_{j+k}\cdots f_{j+1}\otimes b_k(sa_{j+k},\dots,sa_{j+1}).$$
Applying $b_{n-k+1}^a$ yields:
\footnotesize{
$$b_{n-k+1}^a(f_n\otimes sa_n,\dots,f_{j+k+1}\otimes sa_{j+k+1},-f_{j+k}\cdots f_{j+1}\otimes b_k(sa_{j+k},\dots,sa_{j+1}),f_j\otimes sa_j,\dots,f_1\otimes sa_1)$$
$$=(f_n\circ \cdots\circ f_1)\otimes b_{n-k+1}(sa_n\otimes \dots\otimes sa_{j+k+1}\otimes b_k(sa_{j+k},\dots,sa_{j+1})\otimes sa_{j}\otimes \dots\otimes sa_1)$$
}\normalsize
Summing over the indices with the appropriate signs (from the Koszul-Quillen sign rule) reduces the claim to the corresponding statement for $\mathcal{A}$.
\end{proof}

\subsection{Twisted modules}

A \emphbf{pretwisted module}\index{pretwisted module} over $\mathcal{A}$ is a pair $(X,\delta)$, where $X\in \add \mathcal{A}$ and $s\delta\in (s\add \mathcal{A}(X,X))_0$.

Let $(X,\delta)$ be a pretwisted module. Suppose $s\delta=\sum f_k\otimes sa_k$, where $f_k: X(\mathtt{i}_k)\to X(\mathtt{j}_k)$ and $sa_k\in s\mathcal{A}(\mathtt{i}_k,\mathtt{j}_k)$. Then a pretwisted submodule $(X',\delta')$ is defined by the following: A family of subspaces $X'(\mathtt{i})\subseteq X(\mathtt{i})$ such that $f_k(X'(\mathtt{i}_k))\subseteq X'(\mathtt{j}_k)$ for all $k$. If $f_k': X'(\mathtt{i}_k)\to X'(\mathtt{j}_k)$ denotes the restriction of $f_k$ then define $s\delta':=\sum_{k}f_k'\otimes sa_k$. There is an obvious way to define the notion of a pretwisted factor module $(X/X',\delta/\delta')$.

\begin{lem}
The notions of submodule and factor module do not depend on the choice of the presentation.
\end{lem}

\begin{proof}
We only prove the case of submodules. The case of factor modules can be proved analogously. Since $\mathbbm{k}$-vector spaces form a semisimple category it is possible to write $X=X'\oplus Y$. Then $\Hom(X,X)\otimes s\mathcal{A}=(\Hom(X',X')\otimes s\mathcal{A})\oplus (\Hom(X',Y)\otimes s\mathcal{A})\oplus (\Hom(Y,X)\otimes s\mathcal{A})$ by additivity of $\Hom$ and $\otimes$. If $\sum_j g_j\otimes sb_j$ is another presentation of $s\delta=\sum f_k\otimes sa_k$ in  $\Hom(X,X)\otimes s\mathcal{A}$ with $f_k:X'(\mathtt{i}_k)\to X'(\mathtt{j}_k)$, then the part with $X'\to Y$ vanishes and hence one can also assume that $g_j:X'\to X'$. 
\end{proof}

A \emphbf{twisted module}\index{twisted modules $\twmod \mathcal{A}$} over $\mathcal{A}$ is defined to be a pretwisted module $(X,\delta)$ such that:
\begin{enumerate}[{(TM}1{)}]
\item There exists a filtration $(0,0)=(X_0,\delta_0)\subset (X_1,\delta_1)\subset\dots\subset (X_N,\delta_N)=(X,\delta)$ by pretwisted submodules such that the quotients have zero differential, i.e. $(X_i/X_{i-1},$ $\delta_i/\delta_{i-1})=(X_i/X_{i-1},0)$. This condition will be called \emphbf{triangularity} in the remainder.
\item The \emphbf{Maurer-Cartan equation} is satisfied, i.e.
$$\sum_{i=1}^\infty b_i^a(s\delta\otimes \dots\otimes s\delta)=0.$$
\end{enumerate}

Here the first condition guarantees that the second condition makes sense, i.e. that the sum is finite. The second condition will guarantee that $\twmod \mathcal{A}$ is an $A_\infty$-category for the following notion of morphisms:

The morphisms from $(X,\delta_X)$ to $(Y,\delta_Y)$ are defined by $\twmod\mathcal{A}((X,\delta_X),(Y,\delta_Y)):=\add\mathcal{A}(X,Y)$. For $t_i\in \twmod \mathcal{A}((X_{i-1},\delta_{X_{i-1}}), (X_i,\delta_{X_i}))$ the multiplications are given by
$$b_n^{tw}(st_n\otimes \dots\otimes st_1):=\sum_{i_0,\dots,i_n\geq 0}b_{i_0+\dots+i_n+n}^a(s\delta^{\otimes i_n}_{X_n}\otimes st_n\otimes s\delta^{\otimes i_{n-1}}_{X_{n-1}}\otimes \dots\otimes st_1\otimes s\delta^{\otimes i_0}_{X_0}).$$

\begin{lem}
The operations defined above define an $A_\infty$-category structure on $\operatorname{twmod}\mathcal{A}$.
\end{lem}

\begin{proof}
A proof can be found e.g. in \cite[(6.1.2)]{L-H03}.
\end{proof}

By the theorem of Kadeishvili (see Theorem \ref{Kadeishvili}) $H^*(\twmod\mathcal{A})$ carries an $A_\infty$-structure. And Keller and Lef\`evre-Hasegawa have proven the following theorem:

\begin{thm}[{\cite[(7.7)]{K01}}, see also {\cite[Appendix B, Theorem 3.1]{Mad02}}]\label{kellerlefevre}
Suppose we have an $\mathtt{n}$-tuple of $A$-modules $\Delta=(\Delta(\mathtt{1}),\dots\Delta(\mathtt{n}))$ for some finite dimensional algebra $A$. Let $\mathcal{A}=\Ext^*_A(\Delta,\Delta)$ be the $A_\infty$-category given by the extensions. Then $\mathcal{F}(\Delta)\cong H^0(\twmod\mathcal{A})$.
\end{thm}

\section{Filtered modules as representations of quivers}\label{section6}

In this section we dualise the construction given by Keller and Lef\`evre-Hasegawa. This makes it easier to later translate it to the box setting. Since this requires the compatibility of taking tensor products and dualising we need the following additional assumption:

\emphbf{Throughout this section let $\mathcal{A}$ be an $A_\infty$-category such that the spaces $\mathcal{A}(\mathtt{i},\mathtt{j})$ are (totally) finite dimensional for all $\mathtt{i},\mathtt{j}\in \mathcal{A}$}.

\subsection{The convolution category $\conv \mathcal{A}$ }
\label{subsectionconvolutioncategory}

 We now introduce the $A_\infty$-category $\conv  \mathcal{A}$\index{convolution category $\conv \mathcal{A}$} as follows:
\begin{itemize}
\item The objects coincide with those of $\add \mathcal{A}$, i.e. are the $\mathbb{L}$-modules.
\item As in the definition of $(\add \mathcal{A})(X,Y)$ for $X,Y\in \conv \mathcal{A}$ we define a graded $\mathbb{L}\otimes \mathbb{L}$-module $(\conv \mathcal{A})(X,Y)$ via:
$$((\conv \mathcal{A})(X,Y))(\mathtt{i},\mathtt{j})_k:=\Hom_{\mathbbm{k}}((\mathbb{D}\mathcal{A})(\mathtt{i},\mathtt{j})_{-k},\Hom_{\mathbbm{k}}(X(\mathtt{i}),Y(\mathtt{j})))$$
for $\mathtt{i},\mathtt{j}\in \mathcal{A}$, the corresponding graded vector space being obtained by summing up over all $\mathtt{i}$ and $\mathtt{j}$. Note that again $\Hom_{\mathbbm{k}}(X(\mathtt{i}),Y(\mathtt{j}))$ has degree $0$. To see how the shift acts, compute:
\begin{align*}
((s\conv\mathcal{A})(X,Y)(\mathtt{i},\mathtt{j}))_k&=(\conv\mathcal{A}(X,Y)(\mathtt{i},\mathtt{j}))_{k-1}\\
&=\Hom_{\mathbbm{k}}((\mathbb{D}\mathcal{A})(\mathtt{i},\mathtt{j})_{-(k-1)},\Hom_{\mathbbm{k}}(X(\mathtt{i}),Y(\mathtt{j})))\\
&=\Hom_{\mathbbm{k}}(D(\mathcal{A}(\mathtt{i},\mathtt{j})_{k-1}),\Hom(X(\mathtt{i}),Y(\mathtt{j})))\\
&=\Hom_{\mathbbm{k}}(D((s\mathcal{A})(\mathtt{i},\mathtt{j})_k),\Hom(X(\mathtt{i}),Y(\mathtt{j})))
\end{align*}
Thus, the shift just operates on $\conv (\mathcal{A})$ as it does on $\mathcal{A}$ on the right hand side.
\item \label{definitionofd} Let $d_n:=\mathbb{D}b_n:\mathbb{D}s\mathcal{A}\to \mathbb{D}s\mathcal{A}^{\otimes n}$, i.e. the map defined by $d_n(\chi)(sa_n\otimes \dots\otimes sa_1)=\chi(b_n(sa_n\otimes \dots\otimes sa_1))$. Using Sweedler notation we write
 $d_n(\chi)=\sum_{(\chi)}\chi_{(n)}\otimes \dots\otimes \chi_{(1)}$. Let $\overleftarrow{\chi}:=(\chi_{(n)},\dots,\chi_{(1)})$ and similarly $\overleftarrow{sa}:=(sa_n,\dots,sa_1)$. Then we get
\begin{align*}d_n(\chi)(sa_n\otimes\dots\otimes sa_1)&=\sum_{(\chi)}(\chi_{(n)}\otimes \dots\otimes \chi_{(1)})(sa_n\otimes \dots\otimes sa_1)\\&=\sum_{(\chi)}(-1)^{\varepsilon(\overleftarrow{\chi},\overleftarrow{sa})}\chi_{(n)}(sa_n)\otimes\dots\otimes \chi_{(1)}(sa_1),\end{align*}
where $\varepsilon(\overleftarrow{\chi},\overleftarrow{sa})=\sum_{1\leq i<j\leq n}|\chi_{(i)}|\cdot |sa_j|$ as before. Furthermore for homogeneous $sF_i:\mathbb{D}s\mathcal{A}(\mathtt{k}_{i-1},\mathtt{k}_{i})\to \Hom_{\mathbbm{k}}(X(\mathtt{k}_{i-1}),X(\mathtt{k}_{i}))$ define $[sF_n\otimes \dots\otimes sF_1]$ to be the map given by $\chi_n\otimes \dots\otimes \chi_1\mapsto (-1)^{\varepsilon(\overleftarrow{\chi},\overleftarrow{sF})+1}sF_n(\chi_n)\otimes \dots\otimes sF_1(\chi_1)$, where $\overleftarrow{sF}=(sF_n,\dots,sF_1)$. Let $v_n:\overleftarrow{\bigotimes_{i=0}^{n-1}}\Hom_{\mathbbm{k}}(X(\mathtt{k}_{i}),X(\mathtt{k}_{i+1}))\to \Hom_{\mathbbm{k}}(X(\mathtt{k}_0),X(\mathtt{k}_n))$ be the composition map\index{composition map $v_n$ } given by $v_n(f_n\otimes \dots\otimes f_1)=f_n\circ \cdots\circ f_1$. Define $b_n^c(sF_n,\dots,sF_1)$ as the composition
\begin{align*}
\mathbb{D}s\mathcal{A}(\mathtt{k}_0,\mathtt{k}_n)&\stackrel{d_n}{\to} \bigoplus_{\mathtt{k}_1,\dots,\mathtt{k}_{n-1}\in \mathbb{D}s\mathcal{A}}\overleftarrow{\bigotimes_{i=0}^{n-1}}\mathbb{D}s\mathcal{A}(\mathtt{k}_{i},\mathtt{k}_{i+1})\\
&\stackrel{[sF_n\otimes \dots\otimes sF_1]}{\longrightarrow} \bigoplus_{\mathtt{k}_1,\dots,\mathtt{k}_{n-1}\in \mathbb{D}s\mathcal{A}}\bigotimes_{i=0}^{n-1} \Hom_{\mathbbm{k}}(X(\mathtt{k}_{i}),X(\mathtt{k}_{i+1}))\\
&\stackrel{v_n}{\to}\Hom_{\mathbbm{k}}(X(\mathtt{k}_0),X(\mathtt{k}_n)).
\end{align*}
\end{itemize}

Let $V,W$ be $\mathbbm{k}$-vector spaces and let $V$ be finite dimensional. Then there exists a natural isomorphism of vector spaces\index{M@$M$}
$$M:W\otimes_{\mathbbm{k}} V\cong \Hom_{\mathbbm{k}}(V^*,W)$$
where $M(w\otimes v)(\chi)=\chi(v)w$. In particular this gives an isomorphism of graded vector spaces $M_{X,Y}: (\add \mathcal{A})(X,Y)\to (\conv  \mathcal{A})(X,Y)$.

\begin{lem}\label{isofunctor}
We have $Mb_n^a=b_n^cM^{\otimes n}$. In particular the mappings $b_n^c$ endow $\conv \mathcal{A}$ with the structure of an $A_\infty$-category such that $M$ defines an $A_\infty$-isofunctor between the $A_\infty$-categories $\add \mathcal{A}$ and $\conv  \mathcal{A}$.
\end{lem}

\begin{proof}
Fix $X_1,\dots,X_n\in \add \mathcal{A}$ and morphisms $f_i\otimes sa_i\in (\add s\mathcal{A})(X_{i+1},X_{i})$. We will apply $Mb_n^a$ and $b_n^cM^{\otimes n}$ to $(f_n\otimes sa_n)\otimes \dots\otimes (f_1\otimes sa_1)$ and compare the results, which are elements of $\Hom_{\mathbbm{k}}(\mathbb{D}s\mathcal{A},\Hom_{\mathbbm{k}}(X_n,X_1))$. Hence we immediately apply them to some $\chi\in \mathbb{D}s\mathcal{A}$. For $Mb_n^a$ we obtain:
\begin{align*}
Mb_n^a(f_n\otimes sa_n,\dots,f_1\otimes sa_1)(\chi)&=M(-f_n\dots f_1\otimes b_n(sa_n\otimes \dots\otimes sa_1))(\chi)\\
&=-\chi(b_n(sa_n\otimes \dots\otimes sa_1))f_n\cdots f_1
\end{align*}
To apply $b_n^cM^{\otimes n}$ we will use Sweedler notation, i.e. we write $d_n(\chi)=\sum \chi_{(n)}\otimes \dots\otimes \chi_{(1)}$. The entries of $\overleftarrow{M(f\otimes sa)}$ and $\overleftarrow{sa}$ have the same degrees. Thus:\footnotesize
\begin{align*}
b_n^cM^{\otimes n}(f_n\otimes sa_n,\dots,f_1\otimes sa_1)(\chi)&=b_n^c(M(f_n\otimes sa_n),\dots,M(f_1\otimes sa_1))(\chi)\\
&=v_n\circ [M(f_n\otimes sa_n)\otimes \dots\otimes M(f_1\otimes sa_1)](\sum \chi_{(n)}\otimes \dots\otimes \chi_{(1)})\\
&=v_n(\sum_{(\chi)}(-1)^{\varepsilon(\overleftarrow{\chi},\overleftarrow{sa})+1}M(f_n\otimes sa_n)(\chi_{(n)})\otimes \dots\otimes M(f_1\otimes sa_1)(\chi_{(1)}))\\
&=\sum_{(\chi)}(-1)^{\varepsilon(\overleftarrow{\chi},\overleftarrow{sa})+1}v_n(\chi_{(n)}(sa_n)f_n\otimes\dots\otimes \chi_{(1)}(sa_1)f_1)\\
&=\sum_{(\chi)}(-1)^{\varepsilon(\overleftarrow{\chi},\overleftarrow{sa})+1} \chi_{(n)}(sa_n)\cdots \chi_{(1)}(sa_1)f_n \dots f_1\\
&=-d_n(\chi)(sa_n,\dots,sa_1)f_n \cdots f_1\\
&=-\chi(b_n(sa_n\otimes \dots\otimes sa_1))f_n\dots f_1
\end{align*}\normalsize
The other statements follow from the bijectivity of $M$.
\end{proof}

\subsection{Representations of $\mathbbm{k}$-quivers}

 The following lemma is the obvious $\mathbbm{k}$-analogue of the correspondence between representations of quivers and modules over the corresponding  path algebra (viewed as a $\mathbbm{k}$-category). Via this lemma we will later regard a pretwisted module $(X_0,\delta)$ as a module for the path algebra of the quiver $Q^0=(\mathbb{D}s\mathcal{A})^0$ by first applying $M$ as defined in the previous subsection to get $(X_0,X_1:=M(s\delta))$ and then using this as the input of the following lemma:

\begin{lem}\label{representationsmodules}
Let $Q$ be a $\mathbbm{k}$-quiver. Let $\mathbb{L}[Q]$ be the tensor category. Then there is a one-to-one correspondence between modules over $\mathbb{L}[Q]$ and pairs consisting of a functor $X_0:\mathbb{L}\to \mathbbm{k}-\Mod$ (the restriction of the module to $\mathbb{L}$) and an $\mathbb{L}$-bimodule morphism $X_1:Q\to \Hom_{\mathbbm{k}}(X_0,X_0)$ (the restriction to $Q$, the $\mathbb{L}-\mathbb{L}$-bimodule structure of $\Hom_{\mathbbm{k}}(X_0,X_0)$ is given by $(\mathtt{i},\mathtt{j})\mapsto \Hom_{\mathbbm{k}}(X_0(\mathtt{i}),X_0(\mathtt{j}))$ here):
Denote by $v_k$ the composition map given by 
\[v_k:\Hom_{\mathbbm{k}}(X_0,X_0)^k\to \Hom_{\mathbbm{k}}(X_0,X_0), f_k\otimes\cdots \otimes  f_1\mapsto f_k\circ \cdots \circ f_1.\]
 Then $X$ is given on $Q^{\otimes k}$ by
$$X_k:Q^{\otimes k}\to \Hom_{\mathbbm{k}}(X_0,X_0), X_k=v_k\circ X_1^{\otimes k}$$
via $X=\sum_{k=1}^{\infty}X_k$ on the restriction of $X$ to $\bigoplus_{k=1}^{\infty}Q^{\otimes k}$.
\end{lem}

By Gabriel's Theorem, every basic $\mathbbm{k}$-algebra $B$ is isomorphic to $\mathbb{L}[Q]/I$ (regarded as an algebra) for some finite quiver $Q$ and some admissible ideal $I$ (spanned by so-called relations). The foregoing lemma dealt with modules over $\mathbb{L}[Q]$. To include the relations fix a set of generators $\rho_1,\dots,\rho_r$ of the ideal $I$, such that $B\cong \mathbb{L}[Q]/I$. Then $X$ as described above belongs to $B-\Mod$ iff for any $i$ we have $X(\rho_i)=0$.

A representation $X\in \mathbb{L}[Q]-\Mod$ is called \emphbf{rational} if there exists $N\geq 0$ such that $X_N=0$.

\subsection{Convolutional presentation of twisted modules}

Let $Q^0=(\mathbb{D}s\mathcal{A})^0$. Then a \emphbf{pretwisted convolutional module} is a module over $\mathbb{L}[Q^0]$. As remarked in the previous subsection, modules over $\mathbb{L}[Q^0]$ are in one-to-one correspondence to pairs $(X_0,X_1)$, where $X_0:\mathbb{L}\to \mathbbm{k}-\Mod$ is a functor and $X_1:Q^0\to \Hom(X_0,X_0)$ is an $\mathbb{L}$-bimodule morphism. Using Lemma \ref{isofunctor} there is a one-to-one correspondence between those pairs and pretwisted modules (the converse map being given by mapping a pretwisted module $(X_0,\delta)$ to $(X_0,M(s\delta))$). Denote the $\mathbb{L}[Q^0]$-module corresponding to the pretwisted module $(X_0,\delta)$ by $\mathfrak{X}_\delta$.

The notion of a twisted module can be translated in a similar way to obtain the notion of a twisted convolutional module, by the following proposition. This sets up a correspondence between the objects of the category $H^0(\twmod \mathcal{A})$ and the $B$-modules for a certain algebra $B$. To get an  equivalence of categories a definition of the functor on morphisms is needed. This will be postponed to Section \ref{section8}.

\begin{prop}\label{twistedtranslation}
Let $(X_0,\delta)$ be a pretwisted module over $\mathcal{A}$, and (by slight abuse of notation, dropping the $X_0$) $\mathfrak{X}_\delta$ be the corresponding functor $\mathbb{L}[Q^0]\to \mathbbm{k}-\Mod$. Write $s\delta=\sum f_k\otimes sa_k$ with the $sa_k$ being linearly independent. Then the following hold:
\begin{enumerate}[(i)]
\item Let $\mathtt{i},\mathtt{j}\in \mathcal{A}$, $\chi\in Q^0(\mathtt{i},\mathtt{j})$, $x\in X_0(\mathtt{i})$. Then
$$\mathfrak{X}_\delta(\chi)(x)=\sum_k \chi(sa_k) f_k(x).$$
\item $\delta=0$ if and only if $\mathfrak{X}_\delta$ is a semisimple module.
\item A pretwisted module $(X_0',\delta')$ is a pretwisted submodule of $(X_0,\delta)$ iff $\mathfrak{X}_{\delta'}$ is a submodule of $\mathfrak{X}_\delta$.
\item If $(X_0',\delta')\subseteq (X_0,\delta)$, then $\mathfrak{X}_{\delta/\delta'}=\mathfrak{X}_\delta/\mathfrak{X}_{\delta'}$.
\item $\delta$ satisfies the triangularity condition iff $\mathfrak{X}_\delta$ is a rational module over $\mathbb{L}[Q^0]$.
\item $\delta$ satisfies the Maurer-Cartan equation iff $\mathfrak{X}_\delta\circ d=0$, where $d=(d_n)$ as constructed in subsection \ref{definitionofd}. In other words, iff $\mathfrak{X}_\delta$ is a $B:=\mathbb{L}[Q^0]/(\image d)$-module.
\end{enumerate}
\end{prop}

\begin{proof}
\begin{enumerate}[(i)]
\item The first statement just follows by recalling the definition of $M$.
\item Note that a module $Y:\mathbb{L}[Q^0]\to \mathbbm{k}-\Mod$ is a semisimple module iff $Y_1=0$. Write $s\delta=\sum f_k\otimes sa_k$. We will evaluate special 'characters' $\chi_k$ defined by $\chi_k(a_{k'})=\delta_{k,k'}$, the Kronecker delta. By (i) we have $\mathfrak{X}_\delta(\chi_k)=f_k$. Hence $(\mathfrak{X}_\delta)_1=0$ iff $f_k=0$  for all $k$ and hence iff $\delta=0$.
\item Write $s\delta(\mathtt{i},\mathtt{j})=\sum f_k\otimes sa_k$ and $s\delta'(\mathtt{i},\mathtt{j})=\sum f_k'\otimes sa_k$. By (i) for $x\in X'(\mathtt{i})$ (suppressing the inclusions) $\mathfrak{X}_\delta(\chi)(x)=\sum\chi(sa_k) f_k(x)$ and $\mathfrak{X}_{\delta'}(\chi)=\sum \chi(sa_k) f_k'(x)$. For one direction note that if $s\delta'\subseteq s\delta$, then $f_k=f_k'$ for $x\in X'(\mathtt{i})$, hence $\mathfrak{X}_\delta(\chi)(x)=\mathfrak{X}_{\delta'}(\chi)(x)$ for $x\in X'$ and hence $\mathfrak{X}_{\delta'}$ is a subrepresentation of $\mathfrak{X}_\delta$. For the reverse direction note that if $\mathfrak{X}_{\delta'}$ is a subrepresentation of $\mathfrak{X}_{\delta}$, then evaluating them on $\chi_{k}$ yields $f_k=f_k'$ on $X'$. Hence $(X_0',s\delta')$ is a pretwisted submodule of $(X_0,s\delta)$.
\item This is proved analogously.
\item The preceding item shows that the triangularity condition is equivalent to the existence of a composition series. Hence to rationality.
\item For the last claim recall that $[sF_n\otimes \cdots\otimes sF_1]$ was defined by $[sF_n\otimes \cdots\otimes sF_1](\chi_n\otimes \cdots\otimes \chi_1)=(-1)^{\varepsilon(\overleftarrow{\chi},\overleftarrow{sF})+1}sF_n(\chi_n)\otimes \cdots\otimes sF_1(\chi_1)$. Note that because $s\delta\in (\mathbb{D}s\mathcal{A})^0$ we have that $\varepsilon(\overleftarrow{\chi},\overleftarrow{M(f\otimes sa)})=0$. Hence we can replace $[sF_n\otimes\dots\otimes sF_1]$ in the definition of $b_n^c$ by $-(sF_n\otimes\dots\otimes  sF_1)$. Furthermore, since $M$ is an isomorphism,  $\sum_n b_n^a((s\delta)^{\otimes n})=0$ iff
\begin{align*}
0&=\sum_{n=1}^\infty Mb_n^a((s\delta)^{\otimes n})=\sum_n b_n^c M^{\otimes n}((s\delta)^{\otimes n})\\
&=-\sum_n v_n((M(s\delta))^{\otimes n})d_n=-\sum_n (\mathfrak{X}_\delta)_nd_n=-\mathfrak{X}_\delta d
\end{align*}
by Lemma \ref{isofunctor}, definition of $b_n^c$ and the previous subsection.
\end{enumerate}
\end{proof}

Hence we define a \emphbf{twisted convolutional module} to be a rational module over $B$, where $B$ is as defined in the foregoing proposition.

Thus there is a possibility to present the objects of the category of filtered modules as the objects of a category of representations of a quiver with relations. Unfortunately the morphisms cannot just be module homomorphisms since the category of filtered modules is usually not abelian. The next section introduces the notion of a box which will enable us to take care also of the morphisms in this category.

\section{Boxes and their representations}\label{section7}

\subsection{Definition of box representations}

In this section we recall the notion of a box, formerly called bocs. "The term bocs is an approximate acronym for \underline{b}imodule \underline{o}ver \underline{c}ategory with \underline{c}oalgebra \underline{s}tructure" (cf. \cite[2.3 Remarks]{Bur05}). Here we follow the convention of replacing bocs by its existing homonym box as proposed in \cite{DO00}. In our situation the underlying category of the box will be the path category of a $\mathbbm{k}$-quiver with relations, so we can instead regard it as an algebra. The category $\mathbb{L}[Q]/I$ which is isomorphic (as an algebra) to a basic algebra $B$ is sometimes called the spectroid associated to $B$.  For a general introduction to boxes in the context of Drozd's tame and wild dichotomy theorem, see e.g. \cite{D80, CB88, D01}. The theory of left and right algebras we are using is contained in \cite{BB91} or in the nice unpublished manuscript \cite{Bur05}. For a more recent general introduction in a slightly different language, see e.g. \cite{BSZ09}.

\begin{defn}
A \emphbf{box}\footnote{In this case $W$ is also called a $B$-coring by some authors, see e.g. \cite{Brz13}.}\index{box $\mathfrak{B}=(B,W,\mu,\varepsilon)$} $\mathfrak{B}:=(B,W,\mu,\varepsilon)$ or simply $\mathfrak{B}=(B,W)$ consists of a category $B$, a $B$-$B$-bimodule $W$, a coassociative $B$-bimodule comultiplication $\mu:W\to W\otimes_B W$ and a corresponding $B$-bilinear counit $\varepsilon:W\to B$.
A box not necessarily having a counit will be called a \emphbf{prebox}.
\end{defn}

Note that in this definition $W$ is not assumed to have any degree, so especially the Koszul-Quillen sign rule does not apply.

For a box $\mathfrak{B}$ we can define its category of representations. Unlike most categories of representations however, this category will in general not be abelian; this is an advantage for our purposes.

\begin{defn}\index{box representations}
The category $\mathfrak{B}-\Mod$ (respectively $\mathfrak{B}-\modu$) of \emphbf{representations} (respectively \emphbf{locally finite dimensional representations}) of a box $\mathfrak{B}$ is defined as follows:
\begin{enumerate}[(i)]
\item Objects of $\mathfrak{B}-\Mod$ (respectively $\mathfrak{B}-\modu$) are $B$-modules (respectively locally finite dimensional $B$-modules).
\item Morphisms between box representations $X,Y\in \mathfrak{B}-\Mod$ are given by
$$\Hom_{\mathfrak{B}}(X,Y)=\Hom_{B\otimes B^{op}}(W,\Hom_{\mathbbm{k}}(X,Y)),$$
and for morphisms $f\in \Hom_{\mathfrak{B}}(X,Y)$ and $g\in \Hom_{\mathfrak{B}}(Y,Z)$ their composition $gf$ in $\mathfrak{B}-\Mod$ is the morphism of box representations obtained by composing the following $B$-$B$-bimodule homomorphisms:
$$W\stackrel{\mu}{\rightarrow} W\otimes_B W\stackrel{g\otimes f}{\rightarrow}\Hom_\mathbbm{k}(Y,Z)\otimes_B\Hom_{\mathbbm{k}}(X,Y)\stackrel{v_2}{\rightarrow}\Hom_{\mathbbm{k}}(X,Z).$$
\end{enumerate}
\end{defn}

(Here, $v_2$ is the composition map defined in 
\ref{subsectionconvolutioncategory}.)

The associativity of the composition in $\mathfrak{B}-\Mod$ follows from the coassociativity of $\mu$. The identity morphism in $\Hom_{\mathfrak{B}}(X,X)$ can be defined via the following composition of bimodule homomorphisms (because of the counit property):
$$W\stackrel{\varepsilon}{\rightarrow} B\stackrel{\delta_X}{\rightarrow} \Hom_{\mathbbm{k}}(X,X),$$
where $\delta_X$ is the map which is uniquely determined by $\delta_X(\mathbbm{1}_\mathtt{i})=\mathbbm{1}_{X(\mathtt{i})}$ (via the commutativity of the diagrams for a natural transformation).

This is a slightly altered definition of morphism compared to the traditional one (see e.g. \cite[p. 93]{BB91}). There, a morphism $f:X\to Y$ is given by a homomorphism of $B$-modules $f:W\otimes_B X\to Y$. The composition of two such morphisms of box representations $f$ and $g$ is then defined by composing the following $B$-module homomorphisms:
$$gf:W\otimes_B X\stackrel{\mu\otimes \mathbbm{1}_X}{\rightarrow} W\otimes_B W\otimes_B X\stackrel{\mathbbm{1}_W\otimes f}{\rightarrow} W\otimes_B Y\stackrel{g}{\rightarrow} Z.$$
A standard adjunction gives the canonical isomorphism
\[\Hom_{B\otimes B^{op}}(W,\Hom_{\mathbbm{k}}(X,Y))\cong \Hom_B(W\otimes_B X,Y),\tag{$\diamond$}\label{eq:adjunctionformula}\]
\label{adjunction}and that the two definitions of composition agree via this isomorphism.

\subsection{Projective bimodules and boxes with projective kernel}

Let $B$ be a small category with set of objects $S$ (i.e. a \emphbf{category over $S$}). Then a $B$-bimodule $V$ is called a \emphbf{projective $B$-bimodule} if it is isomorphic to a direct sum of the representable functors $B(\mathtt{i},-)\otimes_{\mathbbm{k}}B(-,\mathtt{j})$. The images $\varphi_k$ of $\mathbbm{1}_{\mathtt{i}}\otimes \mathbbm{1}_{\mathtt{j}}$ in $V$ are called \emphbf{generators} of $V$. In this case we also write $V=\oplus B\varphi_k B$. 

The Yoneda lemma implies the following universal property of projective $B$-bimodules which essentially says that all projective objects in the category of $B$-bimodules are direct summands of projective $B$-bimodules.

\begin{lem}
Let $\varphi_k: \mathtt{i}_k\to \mathtt{j}_k$ be a system of generators of a projective bimodule $\bigoplus B\varphi_k B$. Then for every bimodule $M$ and any set of elements $m_k:\mathtt{i}_k\to \mathtt{j}_k$ in $M$ there exists a unique bimodule homomorphism $f:\bigoplus B\varphi_k B\to M$ with $f(\varphi_k)=m_k$.
\end{lem}

In our setting $B$ will be a finite dimensional basic algebra, regarded as a category via $B\cong \mathbb{L}[Q]/I$. In this case it is well-known that the projective bimodules are exactly the projective objects in the category of bimodules.

An important property of boxes for applying the theory of Burt and Butler in Section \ref{section9} is the following:

\begin{defn}
A box $\mathfrak{B}=(B,W,\mu,\varepsilon)$, such that $B$ has finitely many objects and $\varepsilon$ is surjective, is said to have a \emphbf{projective kernel} if $\overline{W}:=\ker\varepsilon$ is a finitely generated projective bimodule.
\end{defn}

For more information on projective bimodules, see e.g. \cite{CB88}.

\subsection{Differential graded categories and boxes}

We explain here the classical transition from (certain) differential graded categories to boxes. The boxes in Section \ref{section8} will arise in this way.

\begin{lem}\label{dgcprebox}
Let $B$ be a category over $S$ and let $U_1$ be a $B$-bimodule. Assume that the tensor category $U:=\bigoplus_{i=0}^{\infty} U_1^{\otimes i}$, where $U^{\otimes 0}_1=B$, is endowed with a grading, such that $|B|=0$ and $|U_1|=1$. Suppose $U$ is equipped with a differential $d$ (a derivation of degree $1$ that squares to $0$). Denote by $(d(B))$ the $B$-bimodule generated by $d(B)$, set $W:=U_1/(d(B))$, and denote the canonical projection $U_1\to W$ by $\pi$. Then there is a prebox $(B,W,\mu)$, where $\mu$ is induced by $d_1:U_1\to U_1\otimes U_1$, i.e.  $(\pi\otimes_B\pi)d_1=\mu\pi$.
\end{lem}

\begin{proof}
We have $\pi^{\otimes 2}(d_1(b_1ub_2))=\pi^{\otimes 2}(d(b_1)\otimes ub_2+b_1d(u)b_2+b_1u\otimes d(b_2))=b_1\pi^{\otimes 2}(d_1(u))b_2$ by definition of $\pi$. Hence $\mu$ is well-defined and a bimodule homomorphism. To prove coassociativity note that
\begin{align*}
(\mu\otimes 1-1\otimes \mu)\mu\pi&=(\mu\otimes 1-1\otimes \mu)\pi^{\otimes 2} d_1\\
&=(\mu\pi\otimes \pi-\pi\otimes \mu\pi)d_1\\
&=\pi^{\otimes 3}((d_1\otimes 1+1\otimes d_1)d_1)\\
&=\pi^{\otimes 3} d^2|_{U_1}=0.
\end{align*}
Note that the Koszul-Quillen sign rule forces the sign change from $-$ to $+$ as $U_1$ and $d$ are of degree $1$ and hence $(1\otimes d_1)(u\otimes u')=-u\otimes d_1(u')$.
\end{proof}

Now we provide a box (without pre) version of this lemma:

\begin{lem}\label{dgcbox}
Let $B$ be a category over $S$ and let $U_1$ be a $B$-bimodule. Furthermore, let $U_1=U_{\Omega}\oplus \overline{U}$ (as $B$-bimodules), such that $U_\Omega$ is a projective bimodule $U_{\Omega}=\bigoplus_{\mathtt{i}\in \mathbb{L}} B\omega_\mathtt{i} B$, where the $\omega_\mathtt{i}$ are generators which are images of $\mathbbm{1}_{\mathtt{i}}\otimes \mathbbm{1}_{\mathtt{i}}$. Suppose that the following conditions hold:
\begin{enumerate}[{(d}1{)}]
\item $d(\omega_\mathtt{i})=\omega_\mathtt{i}\otimes \omega_\mathtt{i}$,
\item for all $b\in B(\mathtt{i},\mathtt{j})$ we have $d(b)=\omega_\mathtt{j}b-b\omega_\mathtt{i}+\partial b$ for some $\partial b\in \overline{U}$,
\item for all $u\in \overline{U}(\mathtt{i},\mathtt{j})$ we have $d(u)=\omega_\mathtt{j}u+u\omega_\mathtt{i}+\partial u$ for some $\partial u\in \overline{U}\otimes \overline{U}$.
\end{enumerate}
Let $W,\mu$ be as in the preceding lemma. Then there is a $B$-bimodule map $\tilde{\varepsilon}: U_1\to B$ given by $\tilde{\varepsilon}(\omega_\mathtt{i})=\mathbbm{1}_\mathtt{i}$ and $\tilde{\varepsilon}(\overline{U})=0$. It induces a unique $\varepsilon: W\to B$ with  $\tilde{\varepsilon}=\varepsilon\pi$ such that $(B,W,\mu,\varepsilon)$ is a box.
\end{lem}

\begin{proof}
By definition and (d2) we have $\tilde{\varepsilon} (d(B))=0$. Thus $\varepsilon$ is well-defined. To see the property of the counit we compute:
\begin{align*}
(\varepsilon\otimes 1)\mu\pi(u)&=(\varepsilon\otimes 1)\pi^{\otimes 2}d_1(u)=(\tilde{\varepsilon}\otimes \pi)d_1(u)\\
&=(\tilde{\varepsilon}\otimes \pi)(\omega_\mathtt{j}\otimes u+u\otimes \omega_\mathtt{i}+\partial u)=\mathbbm{1}_\mathtt{j}\otimes \pi(u)
\end{align*}
for $u\in \overline{U}(\mathtt{i},\mathtt{j})$ by (d3).
and
$$(\varepsilon\otimes 1)\mu\pi(\omega_\mathtt{i})=(\tilde{\varepsilon}\otimes \pi)(\omega_\mathtt{i}\otimes \omega_\mathtt{i}+r_\mathtt{i})=\mathbbm{1}_\mathtt{i}\otimes \pi(\omega_\mathtt{i})$$
by (d1).
\end{proof}

The next lemma gives a sufficient condition for the resulting box to have a projective kernel, which is important for applying the theory of Burt and Butler in Section \ref{section9}. 

\begin{lem}\label{projectivekernel}
If $\mathfrak{B}=(B,W)$ is a box obtained as in the foregoing lemma and it is additionally assumed that $\overline{U}$ is a projective bimodule, then $\overline{W}:=\ker\varepsilon$ is a projective bimodule.
\end{lem}

\begin{proof}
A direct calculation shows that the properties of $d$ in the preceding lemma imply that also $\partial$ gives $B[\overline{U}]:=\bigoplus_{i\geq 0}\overline{U}^{\otimes i}$ the structure of a differential graded algebra. Our assignment of a box to the differential graded algebra $U$ then coincides (up to signs) with the standard procedure of assigning a box to the differential graded algebra $B[\overline{U}]$ described e.g. in \cite{Roi80} or \cite{Brz13}. There the box $(B,W')$ is defined by setting $W':=B\oplus \overline{U}$ with a generator $x$ of $B$ with the left action given by the left actions on $B$ and $\overline{U}$ and the new right action defined by $(bx+u)*b':=bb'x-b\partial(b')+ub'$. Obviously there is a surjective $B$-bilinear map $W'\to W$. Thus, $\dim W\leq \dim W'$. Furthermore there is a surjective $B$-bilinear map $\pi: U_1\to W'$ with $d(B)\subset \ker(\pi)$. The induced map $W\to W'$ proves that $W\cong W'$ as $B$-bimodules and comparing the comultiplications and the counits one sees that $(B,W)$ and $(B,W')$ are ismorphic as boxes. The box $(B,W')$ obviously has a projective kernel. Thus, also $(B,W)$ has a projective kernel.
\end{proof}

\subsection{Morphisms between box representations in terms of generators}\label{boxmorphismswithgenerators}

In this subsection we are concerned with the question how to describe the morphisms in the category of modules over a box in some sense similar to representations of quivers using the universal property of projective bimodules. This makes it easier in the next section to translate the twisted modules to the box setting. 

 Let $X$, $Y$ be $\mathfrak{B}$-representations, where the box $\mathfrak{B}=(B,W)$ is built from a differential graded tensor category as in the previous subsection (including the assumption that $\overline{U}$ is a projective bimodule). Let $I:=(d(B))$ and let $Q^1$ be the $\mathbbm{k}$-quiver over $S$ generated over $\mathbbm{k}$ by the $\varphi_k$, a set of generators for $U_1$. Then
\begin{align*}
\Hom_{B\otimes B^{op}}(W,\Hom_{\mathbbm{k}}(X,Y))&=\Hom_{B\otimes B^{op}}((\bigoplus B\varphi_k B)/I,\Hom_{\mathbbm{k}}(X,Y))\\
&\cong\{\Phi\in \Hom_{B\otimes B^{op}}(\bigoplus B\varphi_k B, \Hom_{\mathbbm{k}}(X,Y))|\Phi(I)=0\}\\
&\cong\{f\in \bigoplus_{\mathtt{i},\mathtt{j}\in \mathcal{A}}\Hom_{\mathbbm{k}}(Q^1(\mathtt{i},\mathtt{j}),\Hom_{\mathbbm{k}}(X(\mathtt{i}),Y(\mathtt{j})))|\Phi_f(I)=0\},
\end{align*}
where $\Phi_f$ is given by $\Phi_f(b_1\varphi_k b_2)=b_1f(\varphi_k) b_2$ and the claimed isomorphisms are isomorphisms of vector spaces.

\section{Filtered modules as box representation}\label{section8}

Note that for the $A_\infty$-category $\mathcal{A}$ of extensions of $\Delta$ the following properties hold:

\begin{enumerate}[{(E}1{)}]
\item $\mathcal{A}^i=0$ for all $i< 0$,
\item $S$, the underlying set of $\mathbb{L}$, is finite and $\mathcal{A}(\mathtt{i},\mathtt{j})$ is finite dimensional for all $\mathtt{i}$ and $\mathtt{j}$,
\item $\mathcal{A}$ is strictly unital, and
\item $\mathcal{A}$ is augmented.
\end{enumerate}

(E1) and (E2) are immediate by definition and for (E3) and (E4), see e.g. \cite[3.5, 7.7]{K01}

Denote by $Q$ the graded $\mathbbm{k}$-quiver $\mathbb{D}s\mathcal{A}$. By (E1) we have that $Q^i=0$ for $i>1$. Let $T$ be the graded tensor category $\mathbb{L}[Q]$. The following lemma is the crucial observation for translating the category of filtered modules to the box language:

\begin{lem}\label{boxconstruction}
\begin{enumerate}[(i)]
\item $T$ is a differential graded category with respect to $d$ induced by $\mathbb{D}b$ via \cite[Lemme 1.2.1.1]{L-H03}, where $b$ is the differential on $s\mathcal{A}$, and the usual multiplication.
\item The ideal $I$ generated by all $Q^i$ with $i<0$ and all $d(Q^{-1})$ is a differential ideal with respect to the $d$ constructed  above.
\item The factor $U=T/I$ is a differential category freely generated over ${B=\mathbb{L}[Q^0]/(\mathbb{L}[Q^0]\cap I)}$ by the $\mathbbm{k}$-quiver $Q^1$. In particular, $U$ satisfies the conditions of Lemma \ref{dgcprebox} and defines a prebox $\mathfrak{B}:=(B,W,\mu)$, where $W=U_1/(d(B))$, and $\mu$ is induced by $d$.
\item For every $\mathtt{i}\in \mathbb{L}$ denote by $\omega_\mathtt{i}$ the element dual to $s\mathbbm{1}_\mathtt{i}$. Then the augmentation decomposition $Q=\mathbb{D}s\mathbb{L}_{\mathcal{A}}\oplus \mathbb{D}s \ker \eta$ satisfies the assumptions of Lemma \ref{dgcbox}. In particular the prebox $\mathfrak{B}$ defines a box.
\item This box $\mathfrak{B}$ has a projective kernel.
\end{enumerate}
\end{lem}

\begin{proof}
\begin{enumerate}[(i)]
\item This follows immediately from dualising our definition of an $A_\infty$-category via the bar construction as follows: Since $\mathbb{D}\Delta_{\overline{T}s\mathcal{A}}^{co}=\mu_{\overline{T}\mathbb{D}s\mathcal{A}}$ and $0=\mathbb{D}(b^2)=(\mathbb{D}b)^2$ we have that $\mathbb{D}b$ is a differential and hence $T$ is a differential graded category.
\item Since $d$ has degree $1$ this is by construction.
\item Since $d$ is of degree $1$, $d(I)\subseteq \mathbb{L}[Q^0]$. Hence freeness follows.
\item Since $\mathcal{A}$ is strictly unital and augmented, the assumptions of Lemma \ref{dgcbox} are satisfied.
\item This is just Lemma \ref{projectivekernel}.
\end{enumerate}
\end{proof}

Combining all the results, the following theorem establishes the description of filtered modules that we need:

\begin{thm}
Let $A$ be a quasi-hereditary algebra with set of standard modules $\Delta$. Then there exists a directed box $\mathfrak{B}=(B,W)$ such that $\mathfrak{B}-\modu \simeq \mathcal{F}(\Delta)$.
\end{thm}

\begin{proof}
By the Theorem of Keller and Lef\`evre-Hasegawa (Theorem \ref{kellerlefevre}) it suffices to prove that $H^0(\twmod \mathcal{A})\cong  \mathfrak{B}-\modu$, where $\mathcal{A}=\Ext^*(\Delta,\Delta)$, regarded as an $A_\infty$-category.
Let $\mathfrak{B}$ be the box constructed in the previous lemma (with $\mathcal{A}$ as the input). On objects the definition of the functor $F:H^0(\twmod \mathcal{A})\to  \mathfrak{B}-\modu$ has been given in Section \ref{section6}, by performing the following steps: First note that an object in  $H^0(\twmod \mathcal{A})$ is the same as an object in the $A_\infty$-category $\twmod \mathcal{A}$, i.e. a pair $(X_0,\delta)$, where $\delta$ satisfies (TM1) and (TM2). Using Lemma \ref{isofunctor}, these pairs are in one-to-one-correspondence with pairs $(X_0,M(s\delta))$, where $\mathbb{L}$-module $X_0:Q_0\to \mathbbm{k}-\modu$ is a functor and $X_1:=M(s\delta):Q^0\to \Hom_{\mathbbm{k}}(X_0,X_0)$ is an $\mathbb{L}$-bimodule isomorphism, where $Q^0=(\mathbb{D}s\mathcal{A})^0$. Using Lemma \ref{representationsmodules}, this corresponds to a module over the path algebra $\mathbb{L}[Q^0]$, which we denoted $\mathfrak{X}_\delta$. Using Proposition \ref{twistedtranslation}, the conditions (TM1) and (TM2) amount to saying that $\mathfrak{X}_\delta$ is in fact a $B$-module (where $B$ is as constructed there). As $B$ coincides with the underlying algebra of the box $\mathfrak{B}$, the composition of these steps provides the definition of the functor $F$ on objects.\\
For the morphisms first note that $H^0(\operatorname{twmod}\mathcal{A})=Z^0(\operatorname{twmod}\mathcal{A})$ since $B^0(\operatorname{twmod}\mathcal{A})=0$ (this follows from the definition of $\twmod \mathcal{A}$ as there is no $\Ext^*$ in negative degrees). Thus a morphism in $H^0(\twmod \mathcal{A})$ is an $f\in (\add \mathcal{A}(X,Y))_0$ with $b_1^{tw}(sf)=0$. Using again the isofunctor $M$ this corresponds to an $\mathbb{L}$-bimodule homomorphism $M(sf)\in \Hom_{\mathbbm{k}}(Q^1,\Hom_{\mathbbm{k}}(X,Y))$. It remains to prove that the condition $b_1^{tw}(sf)=0$ is equivalent to $\Phi_f(I)=0$. In fact, by the results of subsection \ref{boxmorphismswithgenerators}, the latter condition precisely defines a morphism in the category of modules for $\mathfrak{B}$. The following calculation shows that indeed $b_1^{tw}(sf)=0$ is equivalent to $\Phi_f(I)=0$:
\begin{align*}
0&=Mb_1^{tw}(sf)=M\sum_{i_0,i_1\in \mathbb{N}_0}b_{1+i_0+i_1}^a((s\delta)^{\otimes i_1}\otimes sf\otimes (s\delta)^{\otimes  i_0})\\
&=\sum b_{1+i_0+i_1}^cM^{\otimes (1+i_0+i_1)} ((s\delta)^{\otimes i_1}\otimes sf\otimes (s\delta)^{\otimes i_0})\\
&=\sum v_{1+i_0+i_1}[M(s\delta)^{\otimes i_1}\otimes Mf\otimes M(s\delta)^{\otimes i_0}]d_{1+i_0+i_1}\\
&=-\sum (\Phi_f)_{1+i_0+i_1, i_0+1}(d_{1+i_0+i_1})\\
&=-\Phi_f\circ d
\end{align*}
since the degrees of the $s\delta$ are zero, whence they contribute no sign and if the resulting expression is non-zero, the $Mf$ is not 'switched' with any element of non-zero degree. Thus, the composition of these steps defines the functor $F$ on morphisms.\\
 For checking functoriality, the calculation is similar:
\begin{align*}
Mb_2^{tw}(sf,sg)&=\sum Mb_{2+i_0+i_1+i_2}^a((s\delta)^{\otimes i_2}\otimes sf\otimes (s\delta)^{\otimes i_1}\otimes sg\otimes (s\delta)^{\otimes i_0})\\
&=\sum b_{2+i_0+i_1+i_2}^cM^{\otimes (2+i_0+i_1+i_2}((s\delta)^{\otimes i_2}\otimes sf\otimes (s\delta)^{\otimes i_1}\otimes sg\otimes (s\delta)^{\otimes i_0})\\
&=\sum v_{2+i_0+i_1+i_2}[M(s\delta)^{\otimes i_2}\otimes M(sf)\otimes M(s\delta)^{\otimes i_1}\otimes M(sg)\otimes M(s\delta)^{\otimes i_0}]d_{2+i_0+i_1+i_2}
\end{align*}
Again the $M(s\delta)$ do not contribute a sign, and if the expression is non-zero, the $M(sf)$ is 'switched' exactly once with an element of degree $1$. So the sign in the last row given by $[\cdots ]$ compared to the tensor product (without the Koszul-Quillen sign rule) is $+1$, independent of the summand. Thus the two compositions coincide.
\end{proof}

\section{Exact structure on representations of boxes}

In this section we want to construct an exact structure on the category $\mathfrak{B}-\modu$, where $\mathfrak{B}$ is a box constructed by Lemma \ref{boxconstruction}. The general strategy follows \cite{BSZ09} but we have to take extra care because the underlying algebra of our box may have relations. The first lemma in this section describes the category of representations for the special boxes we construct in a way similar to the category of quiver representations.

\begin{lem}
Let $\mathfrak{B}$ be a box constructed in Lemma \ref{boxconstruction}. Let $d(r)=r_{-1}+\partial r$ where $r_{-1}$ is in the ideal generated by $Q^{-1}$ and $\partial r\in \mathbb{L}[Q^0]$.
Then the category $\mathfrak{B}-\modu$ can be described as follows:
\begin{enumerate}[(i)]
\item Objects are given by representations $X$ of $\mathbb{L}[Q^0]$ satisfying $X(\partial(r))=0$ for all $r\in Q^{-1}$.
\item Let $\omega_\mathtt{i}:=\mathbb{D}s\mathbbm{1}_\mathtt{i}$ and let $Q_\Omega$ be the $\mathbbm{k}$-subquiver of $Q^1$ given by the $\omega_i$. Let $\overline{Q}$ be a direct complement of $Q_\Omega$ in $Q^1$. Let $p:\mathbb{L}[Q]\to \mathbb{L}[Q]$ be the map factoring out the ideal spanned by all $Q^{\leq -1}$. Let $X$ and $Y$ be in $\mathfrak{B}-\modu$. Then a morphism $f:X\to Y$ is given by a map $f_q: X(s(q))\to Y(t(q))$ for every $q\in Q^1$ satisfying $\Phi_f(p(d(b)))=0$ for all $b\in Q^0$, where $\Phi_f(b_1\varphi b_2)=b_1f(\varphi)b_2$ for $b_1, b_2\in \mathbb{L}[Q^0]$, $\varphi\in Q^1$. Abbreviate $f_{\omega_\mathtt{i}}$ by $f_\mathtt{i}$.
\item For the composition note that $pd(\varphi)=\omega_\mathtt{j}\varphi+\varphi\omega_\mathtt{i}+\partial \varphi$ for $\varphi\in \overline{Q}(\mathtt{i},\mathtt{j})$ with $\partial \varphi \in \mathbb{L}[Q^0\oplus \overline{Q}]$. Use Sweedler notation to write $\partial \varphi=\sum_{(\varphi)} b_3\varphi_2b_2\varphi_1b_1$ with $b_k\in \mathbb{L}[Q^0]$ and $\varphi_k\in \overline{Q}$. Then, $(fg)_\mathtt{i}=f_\mathtt{i}g_\mathtt{i}$ and $(fg)_\varphi=f_\mathtt{j}g_\varphi+f_\varphi g_\mathtt{i}+\sum_{(\varphi)}b_3f_{\varphi_2}b_2f_{\varphi_1}b_1$.
\end{enumerate}
\end{lem}

\begin{proof}
This follows as in Subsection \ref{boxmorphismswithgenerators}, replacing $B$ with $\mathbb{L}[Q^0]$.
\end{proof}

We assume that the quiver $Q$ is directed. We define a valuation $\nu:Q\to \mathbb{N}_0$ by $\nu(a)=t(a)-s(a)$. In the remainder we will prove many statements by induction on $\nu(a)$. This strategy is similar to the classical situation of free triangular boxes. Also similarly to the classical situation in the following proposition and the lemma thereafter we introduce certain normal forms for epimorphisms and monomorphisms.

\begin{prop}\label{definemorphism}
Let $\mathfrak{B}$ be a box constructed in Lemma \ref{boxconstruction}. Let $X\in \mathfrak{B}-\modu$ be a representation and $(Y_{\mathtt{i}})_{\mathtt{i}=\mathtt{1}}^{\mathtt{n}}$ be $\mathbbm{k}$-vector spaces with vector space isomorphisms $f_\mathtt{i}:X(\mathtt{i})\to Y_\mathtt{i}$ (respectively $g_\mathtt{i}:Y_\mathtt{i}\to X(\mathtt{i})$) and vector space homomorphisms $f_\varphi:X(\mathtt{i})\to Y_\mathtt{j}$ (respectively $g_\varphi:Y_\mathtt{i}\to X(\mathtt{j})$) for any $\varphi\in \overline{Q}(\mathtt{i},\mathtt{j})$. Then there is a unique representation $Y\in \mathfrak{B}-\modu$ and a morphism $f:X\to Y$ (respectively $g:Y\to X$) such that $Y(\mathtt{i})=Y_\mathtt{i}$ with the predefined $f_\mathtt{i}$ (respectively $g_\mathtt{i}$) for all $\mathtt{i}\in Q_0$ and predefined $f_\varphi$ (respectively $g_\varphi$) for all $\varphi\in \overline{Q}$.
\end{prop}

\begin{proof}
We only prove the case $f:X\to Y$. The other case is dual. We define $Y(a)$ inductively by induction on $\nu(a)$. For $\nu(a)=1$ we can define $Y(a):=f_{t(a)}(X(a)+\Phi_f(\partial a))f_{s(a)}^{-1}$. By induction we can assume that $Y(b)$ is defined for all $b\in Q^0$ with $\nu(b)<k$ and that furthermore for any $r\in Q^{-1}$ with $\nu(r)<k$ holds $Y(\partial r)=0$. For $f$ to be a representation we want that 
\begin{align*}
0&=\Phi_f(pd(a))=\Phi_f(\omega_{t(a)}a-a\omega_{s(a)}+\partial a)\\
&=f_{t(a)}X(a)-Y(a)f_{s(a)}+\Phi_f(\partial a),
\end{align*}
where $p:\mathbb{L}[Q]\to \mathbb{L}[Q]$ is as before the map factoring out the ideal spanned by all $Q^{\leq -1}$. Since by induction $\Phi_f$ has been constructed on all constituents of $\partial a$ one can define $Y(a):=(f_{t(a)}X(a)+\Phi_f(\partial a))f_{s(a)}^{-1}$. We have to check that $Y$ defined like this satisfies $Y(\partial r)=0$ for all $r$ with $\nu(r)=k$. We furthermore have that $\Phi_f$ is defined on all constituents of $d(\partial r)$ and $\Phi_f(d(\partial r))=0$. Thus:
\begin{align*}
0&=\Phi_f(pd(\partial r))=\Phi_f(\omega_{t(r)}\partial r-\partial r\omega_{s(r)}+p\partial^2r)\\
&=f_{t(r)}X(\partial r)-Y(\partial r)f_{s(r)}
\end{align*}
Since $X$ is a representation of $\mathfrak{B}$ it satisfies $X(\partial r)=0$, and as the $f_\mathtt{i}$ are isomorphisms $Y(\partial r)=0$ as well.
\end{proof}

\begin{lem}\label{makeinduced}
Let $\mathfrak{B}$ be a box constructed in Lemma \ref{boxconstruction}. Let $f:X\to Y$ be a morphism in $\mathfrak{B}-\modu$ such that for any $\mathtt{i}$ the map $f_\mathtt{i}:X(\mathtt{i})\to Y(\mathtt{i})$ is an epimorphism (respectively a monomorphism). Then there exists a representation $Z\in \mathfrak{B}-\modu$ and a homomorphism $h:Z\to X$ (respectively $h:Y\to Z$), such that $h_\mathtt{i}$ are isomorphisms and $(fh)_\varphi=0$ (respectively $(hf)_\varphi=0$) for all $\varphi\in \overline{Q}$.
\end{lem}

\begin{proof}
We only prove the case of $f_\mathtt{i}$ being epimorphisms. The other statement is dual. Let $Q^1_{\nu\leq k}=\{a\in Q^1|\nu(a)\leq k\}$. Then the claim follows by induction on $k$ from the following:
\begin{description}
\item[Claim] Suppose $f_q=0$ for all $q\in Q^1_{\nu\leq k}$, then there exists a morphism $h:Z\to X$ in $\mathfrak{B}-\modu$ with $h_\mathtt{i}$ an isomorphism and $(fh)_\varphi=0$ for all $\varphi \in Q^1_{\nu\leq {k+1}}(\mathtt{i},\mathtt{j})$.
\end{description}
Let $f_\mathtt{i}'$ be a right inverse of $f_\mathtt{i}$. Now, use the foregoing proposition with $h_\mathtt{i}=\mathbbm{1}_{X(\mathtt{i})}$ and $h_\varphi=-f_{t(\varphi)}'f_\varphi$ to define a representation $Z$ and a morphism $h:Z\to X$ in $\mathfrak{B}-\modu$. Obviously, $h_\mathtt{i}$ is an isomorphism and by using Sweedler notation we have for $\varphi\in Q^1_{\nu\leq k+1}$: 
\[(fh)_\varphi=f_\mathtt{j}h_\varphi+f_\varphi h_\mathtt{i}+\sum_{(\varphi)} b_3 f_{\varphi_2} b_2 h_{\varphi_1} b_1= -f_\mathtt{j}f_{t(\varphi)}'f_\varphi+f_\varphi=0.\]
\end{proof}

\begin{lem}
Let $\mathfrak{B}$ be a box constructed in Lemma \ref{boxconstruction}. A morphism $f:X\to Y$ in $\mathfrak{B}-\modu$ is an isomorphism iff all $f_\mathtt{i}$ are isomorphisms of $\mathbbm{k}$-vector spaces.
\end{lem}

\begin{proof}
If $f$ is an isomorphism in $\mathfrak{B}-\modu$, then there exists $g$ with $fg=\mathbbm{1}_Y$ and in particular $(fg)_\mathtt{i}=f_\mathtt{i}g_\mathtt{i}=\mathbbm{1}_{Y(\mathtt{i})}$. Dually also $g_\mathtt{i}f_\mathtt{i}=\mathbbm{1}_{X(\mathtt{i})}$. Hence $f_\mathtt{i}$ is an isomorphism.\\
Conversely, assume that $f$ is a morphism in $\mathfrak{B}$ with all $f_\mathtt{i}$ being isomorphisms. By the preceding lemma there exists a morphism $h$ with all $h_\mathtt{i}$ being isomorphisms and $(hf)_\varphi=0$ for all $\varphi\in \overline{Q}$. Hence $(hf)_\mathtt{i}$ are isomorphisms and $(hf)_\varphi=0$. Thus $(hf)$ is in the image of the functor $\pi_{\mathfrak{B}}^*: B-\modu\to \mathfrak{B}-\modu$ being the identity on objects and sending a morphism $(l_\mathtt{i})_{\mathtt{i}=\mathtt{1}}^\mathtt{n}$ to the morphism with the same $l_\mathtt{i}$ and $l_\varphi=0$ for all $\varphi\in \overline{Q}$. Thus, $hf$ is the image of an isomorphism and hence, is an isomorphism. Thus, $f$ is a section. Dually, using the other part of the foregoing proposition, one can show that $f$ also is a retraction.
\end{proof}

\begin{prop}
Let $\mathfrak{B}$ be a box constructed in Lemma \ref{boxconstruction}. Then $\mathfrak{B}-\modu$ is fully additive (also called idempotent complete or Karoubian), i.e. if $e\in \Hom_\mathfrak{B}(X,X)$ is an idempotent, then there is a box representation $Y$ and morphisms of box representations $\pi:X\to Y$ and $\iota: Y\to X$ such that $e=\iota\pi$ and $\pi\iota=\operatorname{id}_Y$.
\end{prop}

\begin{proof}
We show that every idempotent $e$ in $\mathfrak{B}-\modu$ is conjugate to an idempotent in the image of the embedding functor $\pi_{\mathfrak{B}}^*: B-\modu\to \mathfrak{B}-\modu$. If this is true, i.e. $e'=h^{-1}eh$ with $e'$ in the image, then $e'$ obviously is an idempotent and splits (first in $B-\modu$, and by applying the embedding functor also in $\mathfrak{B}-\modu$). This means that there exist $f, g$ morphisms in $\mathfrak{B}-\modu$ with $gf=e'$ and $fg=\mathbbm{1}_Y$ for some $Y$. Hence $e=(hg)(fh^{-1})$ and $(fh^{-1})(hg)=\mathbbm{1}_Y$. Thus, $e$ also splits.\\
To show the conjugacy we again use induction and the following claim:
\begin{description}
\item[Claim] Suppose $e_q=0$ for all $q\in Q^1_{\nu\leq k}$, then there exists an isomorphism $h:Z\to X$ in $\mathfrak{B}-\modu$ with $(h^{-1}eh)_\varphi=0$ for all $\varphi\in Q^1_{\nu\leq k+1}$. 
\end{description}
Since $e$ is an idempotent, also $e_\mathtt{i}$ is an idempotent for all $\mathtt{i}$ as $e_{\mathtt{i}}^2=(e^2)_\mathtt{i}=e_\mathtt{i}$. For $\varphi\in Q^1_{\nu\leq k+1}$ we have $e_\varphi=(e^2)_\varphi=e_{t(\varphi)}e_\varphi+e_\varphi e_{s(\varphi)}$. By multiplying with $e_{t(\varphi)}$ and $e_{s(\varphi)}$ from the left and right, respectively, we get $e_{t(\varphi)}e_\varphi e_{s(\varphi)}=0$.\\
Using Proposition \ref{definemorphism} with $h_\mathtt{i}=\mathbbm{1}_{X(\mathtt{i})}$ and $h_\varphi=e_\varphi(2e_{s(\varphi)}-\mathbbm{1}_{X(s(\varphi))})$  we get an isomorphism $h:Z\to X$. Let $g:=h^{-1}:X\to Z$. Then $g_\mathtt{i}=\mathbbm{1}_{X(\mathtt{i})}$ and, since for $\varphi\in Q^1_{\nu\leq k+1}$ we have $0=(gh)_\varphi=g_{t(\varphi)}h_\varphi+g_\varphi h_{s(\varphi)}+\sum_{(\varphi)}b_3g_{\varphi_1} b_2h_{\varphi_2}b_1=h_\varphi+g_\varphi$ we obtain for $\varphi\in Q^1_{\nu\leq k+1}$ that $g_\varphi=-h_\varphi$. Thus, for $\varphi\in Q^1_{\nu\leq k+1}$ we have:
\begin{align*}
(geh)_\varphi&=g_{t(\varphi)}(eh)_\varphi+g_\varphi (eh)_{s(\varphi)}\\
&=g_{t(\varphi)}(e_{t(\varphi)}h_\varphi+e_\varphi h_{s(\varphi)})-h_\varphi e_{s(\varphi)}\\
&=e_{t(\varphi)}e_\varphi(2e_{s(\varphi)}-\mathbbm{1}_{X(s(\varphi))})+e_\varphi-(e_\varphi(2e_{s(\varphi)}-\mathbbm{1}_{X(s(\varphi))}))e_{s(\varphi)}\\
&=-e_{t(\varphi)}e_\varphi+e_\varphi-2e_\varphi e_{s(\varphi)}+e_\varphi e_{s(\varphi)}\\
&=(e-e^2)_\varphi=0.
\end{align*}
\end{proof}

\begin{defn}
Let $\mathcal{C}$ be an additive $\mathbbm{k}$-category with a class of \emphbf{exact pairs}, i.e. pairs of morphisms $(f,g)$ with $g$ a cokernel of $f$ and $f$ a kernel of $g$, in $\mathcal{C}$ closed under isomorphisms. The maps $f$ (resp. the maps $g$) in exact pairs $(f,g)$ are called \emphbf{inflations} (resp. \emphbf{deflations}). Then $\mathcal{C}$ is called an \emphbf{exact category} iff
\begin{enumerate}
\item[(E1)] The composition of deflations is a deflation.
\item[(E2)] For each morphism $Z'\to Z$ and each deflation $d:Y\to Z$, there exists a morphism $f':Y'\to Y$ and a deflation $d':Y'\to Z$ such that $df'=fd'$.
\item[(E3)] Identities are deflations. If $gf$ is a deflation, then so is $g$.
\item[(E3)$^{op}$] Identities are inflations. If $gf$ is an inflation, then so is $f$.
\end{enumerate}
\end{defn}

We will later prove that the category $\mathfrak{B}-\modu$ is an exact category with the following notion of exact pairs:

\begin{defn}
Let $f:X\to Y$, $g:Y\to Z$ be morphisms in $\mathfrak{B}-\modu$. Then $(f,g)$ is called an \emphbf{exact pair} if $gf=0$ and the sequence $0\to X(\mathtt{i})\stackrel{f_\mathtt{i}}{\to} Y(\mathtt{i})\stackrel{g_\mathtt{i}}{\to} Z(\mathtt{i})\to 0$ is (split) exact for every $\mathtt{i}$. \\
Furthermore define an equivalence relation on the exact pairs starting in $X$ and ending in $Z$ by setting $(f,g)\sim (f',g')$ if there exists an isomorphism $h:Y\to Y'$ with $f'=hf$ and $g'=gh^{-1}$.
\end{defn}

\begin{lem}
Let $\mathfrak{B}$ be a box constructed in Lemma \ref{boxconstruction}. Every exact pair is equivalent to one of the form $\pi_\mathfrak{B}^*(E)$ for an exact sequence $E$ in $B-\modu$.
\end{lem}

\begin{proof}
Let $X\stackrel{f}{\to} Y\stackrel{g}{\to} Z$ be an exact pair. Then, by Proposition \ref{definemorphism} there exists an isomorphism $h:Y\to Y'$ such that the following diagram commutes:
\[\begin{xy}
\xymatrix{
X\ar[r]^{f}\ar@<2pt>@{-}[d]\ar@<-2pt>@{-}[d] &Y\ar[r]^{g}\ar[d]^h &Z\ar@<2pt>@{-}[d]\ar@<-2pt>@{-}[d]\\
X\ar[r]^{f'} &Y'\ar[r]^{gh^{-1}} &Z,
}
\end{xy}\]
where $f'$ is a morphism with $f'_\mathtt{i}=f_\mathtt{i}$ for all $\mathtt{i}$ and $f_\varphi=0$ for all $\varphi\in \overline{Q}$. Thus, we can and will assume that $f_\varphi=0$ for all $\varphi\in \overline{Q}$. The proof follows by induction from the following claim:
\begin{description}
\item[Claim] If $g_\varphi=0$ for all $\varphi\in Q^1_{\nu\leq k}$, then the exact pair is equivalent via $h$ to a pair $(h^{-1}f)$ and $(gh)$ with $(h^{-1}f)_\varphi=0$ for all $\varphi \in \overline{Q}$ and $(gh)_\varphi=0$ for all $\varphi\in Q^1_{\nu\leq k+1}$.
\end{description}
Using Proposition \ref{definemorphism} let $h:Y\to Y'$ be the morphism given by $h_\mathtt{i}=\mathbbm{1}_{Y(\mathtt{i})}$ and $h_\varphi=g_\mathtt{i}'g_\varphi$, where $g_\mathtt{i}'$ is a right inverse of $g_\mathtt{i}$. We have to show (a) that $(hf)_\varphi=0$ for all $\varphi\in \overline{Q}$ and (b) that $(gh^{-1})_\varphi=0$ for all $\varphi\in Q^1_{\nu\leq k+1}$.\\
For (a) note that since $0=gf$ we have for $\varphi\in Q^1_{\nu\leq k+1}$ that $0=(gf)_\varphi=g_{t(\varphi)}f_\varphi+g_\varphi f_{s(\varphi)}=g_\varphi f_{s(\varphi)}$. Thus, we have $(hf)_\varphi=h_{t(\varphi)}f_\varphi+h_\varphi f_{s(\varphi)}=g_{t(\varphi)}'g_\varphi f_{s(\varphi)}=0$. For (b) note that since $hh^{-1}=\mathbbm{1}_{Y'}$ we have that $0=(hh^{-1})_\varphi=h_{t(\varphi)}(h^{-1})_\varphi+h_\varphi (h^{-1})_{s(\varphi)}=(h^{-1})_\varphi+h_\varphi$ for all $\varphi\in Q^1_{\nu\leq k+1}$. Thus, we have
\begin{align*}
(gh^{-1})_\varphi&=g_{t(\varphi)}(h^{-1})_\varphi+g_\varphi (h^{-1})_{s(\varphi)}\\
&=-g_{t(\varphi)}h_\varphi+g_\varphi=-g_{t(\varphi)}g'_{t(\varphi)}g_\varphi+g_\varphi=0
\end{align*}
\end{proof}

Next, we show that our notion of an exact pair coincides with the usual notion in exact categories:

\begin{lem}
Let $X\stackrel{f}{\to} Y\stackrel{g}{\to} Z$ be an exact pair in $\mathfrak{B}-\modu$. Then $f$ is a kernel of $g$ and $g$ is a cokernel of $f$.
\end{lem}

\begin{proof}
It suffices to prove that the exact pairs of the form $\pi^*_\mathfrak{B}(E)$ satisfy this property since the notion of being a kernel or a cokernel does not change under isomorphism.\\
So suppose an exact pair $X\stackrel{f}{\to} Y\stackrel{g}{\to} Z$ is induced from $B-\modu$ and suppose there is a morphism $n:N\to Y$ in $\mathfrak{B}-\modu$ with $gn=0$. Then $g_\mathtt{i}n_\mathtt{i}=0$ and $0=(gn)_\varphi=g_{t(\varphi)}n_\varphi+g_\varphi n_{s(\varphi)}+\sum_{(\varphi)}b_3g_{\varphi_2}b_2n_{\varphi_1}b_1=g_{t(\varphi)}n_\varphi$ since $g$ is induced. Thus, since $f_\mathtt{i}$ is a kernel of $g_\mathtt{i}$ there are unique maps $l_\mathtt{i}$ and $l_\varphi$ such that $f_\mathtt{i}l_i=n_i$ and $f_{s(\varphi)}l_\varphi=n_\varphi$. It remains to prove that these $l_i, l_\varphi$ define a morphism in $\mathfrak{B}-\modu$. For this, note that $n$ is a morphism in $\mathfrak{B}-\modu$. Hence if we denote $\partial b= \sum_{(b)} b_2\varphi b_1$ the following equality holds: $0=\Phi_n(pd(b))=f_{t(a)}l_{t(a)}b-bf_{s(a)}l_{s(a)}+\sum_{(b)} b_2 f_{s(\varphi)}l_\varphi b_1$. Since $f$ is in fact a morphism in $B-\modu$, it commutes with elements of $b$. Hence we can multiply with a left inverse of $f$ and get the equality $0=\Phi_l(pd(b))$.  
\end{proof}

\begin{lem}
The deflations (respectively inflations) are exactly those morphisms $g$, such that $g_{\mathtt{i}}$ is an epimorphism (respectively monomorphism) for all $\mathtt{i}$.
\end{lem}

\begin{proof}
We only prove the case of deflations, the other case follows by dual arguments. Let $g:Y\to Z$ be a morphism with all $g_\mathtt{i}$ being epimorphisms. Then by Lemma \ref{makeinduced} there exists an isomorphism $h:Y'\to Y$ with $g':=gh$ being induced and $(gh)_\mathtt{i}=g_\mathtt{i}$ for all $\mathtt{i}$. Let $f'$ be the kernel of $g'$ in $B-\modu$. By the preceding lemma, by embedding $f'$ is also the kernel of $g'$ in $\mathfrak{B}-\modu$. Then defining $f:=hf'$ gives an exact pair $(f,g)$, and hence $g$ is a deflation.
\end{proof}

\begin{thm}
Let $\mathfrak{B}$ be a box constructed in Lemma \ref{boxconstruction}. Then $\mathfrak{B}-\modu$ is an exact category.
\end{thm}

\begin{proof}
The properties (E1), (E3) and (E3)$^{op}$ follow immediately from the foregoing lemma.\\
For (E2) let $X\stackrel{i}{\to} Y\stackrel{d}{\to} Z$ be an exact pair and $f:Z'\to Z$ be a morphism in $\mathfrak{B}-\modu$. Then by the preceding lemma  since $d$ is a deflation, also $(d,f):Y\oplus Z'\to Z$ is a deflation. Thus there exists an exact pair $(\begin{pmatrix}-f'\\d'\end{pmatrix}, (d,f))$. In particular $d_\mathtt{i}f'_\mathtt{i}=f_\mathtt{i}d'_\mathtt{i}$ gives a pullback-diagram in $B-\modu$. Hence $d'$ ($d'_\mathtt{i}$ being parallel to $d_\mathtt{i}$) is a deflation by the foregoing lemma. Thus, we have proved (E2).  
\end{proof}

In the next section we will construct an exact functor $R\otimes_B -$ which allows us to see this exact category as the full subcategory of an abelian category with the exact structure induced from that.

\section{Boxes: Burt-Butler theory}\label{section9}

In this section, we are going to review the relevant part of the results by Burt and Butler; the definition of the left and right Burt-Butler algebras (just called $L$ and $R$ by Burt and Butler), basic properties of these algebras and the description of the category of representations of a box as categories of induced or coinduced modules inside the module categories of the left and right algebras, respectively. We do not need any facts from the sophisticated representation theory of boxes, only the definition and some basic techniques on coalgebras and from homological algebra. All material is taken from \cite{BB91}; another exposition of the same material is contained in \cite{Bur05}. Throughout this section, $B$ is assumed to be a basic finite dimensional algebra. We have already seen that we can regard this as a category by fixing an isomorphism $B\cong \mathbb{L}[Q]/I$.

\subsection{Left and right Burt-Butler algebras}

The left or right $B$-module $B$ is, by definition, a left or right representation of the box $\mathfrak{B}$ (here right representation is defined by the obvious 'dual' of left representation) and as such has endomorphism rings. These turn out to be exactly the algebras we are looking for. These have been studied by Burt and Butler \cite{BB91} under the name left and right algebra of a box, respectively.

\begin{defn}\index{left/right Burt-Butler algebra $L_\mathfrak{B}$, $R_\mathfrak{B}$}
The \emphbf{left Burt-Butler algebra} $L_\mathfrak{B}=L$ of a given box $\mathfrak{B}=(B,W)$ is the endomorphism ring $\End_{\mathfrak{B}^{op}}(B)\cong \Hom_B(W_B,B_B)$ (compare \eqref{eq:adjunctionformula} on page \pageref{adjunction})  of the right module $B$, i.e. with multiplication $e \cdot f: W \stackrel{\mu}{\rightarrow} W \otimes W  \stackrel{f \otimes 1}{\rightarrow} B \otimes W \simeq W \stackrel{e}{\rightarrow} B$ for $e,f\in L$.

The \emphbf{right Burt-Butler algebra} $R_\mathfrak{B}=R$ of a given box $\mathfrak{B}$ is the algebra $\End_{\mathfrak{B}}(B)^{op}\cong \Hom_B({}_BW,{}_BB)$ (compare \eqref{eq:adjunctionformula} on page \pageref{adjunction}) of the left module $B$, i.e. with multiplication $e \cdot f: W \stackrel{\mu}{\rightarrow} W \otimes W \stackrel{1 \otimes e}{\rightarrow} W \otimes B \simeq W
\stackrel{f}{\rightarrow} B$ for $e,f\in L$.

In both cases, $\varepsilon$ (the counit of $\mathfrak{B}$) is the identity element of $L$ and $R$, respectively.
\end{defn}

\subsection{Representations of boxes as induced and coinduced modules}

It is proven in \cite[Proposition 1.2]{BB91} that there is an $L$-$R$-bimodule structure on $W$. This bimodule structure on ${}_LW_R$ defines functors $W \otimes_R -: R-\Mod \rightarrow L-\Mod$ and $\Hom_L(W,-): L-\Mod \rightarrow R-\Mod$, which will be seen to restrict to equivalences between certain subcategories.

\begin{defn}
The category $\Ind(B,R)$ of \emphbf{induced modules}\index{induced modules $\Ind(B,R)$} is the full subcategory of $R-\modu$ whose objects are of the form ${}_RM \simeq R \otimes_B X$ for some
finitely generated $B$-module $X$.

The category $\CoInd(B,L)$ of \emphbf{coinduced modules}\index{coinduced modules $\CoInd(B,L)$} is the full subcategory of $L-\modu$ whose objects are of the form ${}_LM \simeq \Hom_B(L,X)$ for some finitely generated $B$-module $X$.
\end{defn}

\begin{prop}[{\cite[Proposition 2.2]{BB91}}]\label{exact}
Let $\mathfrak{B}$ be a box with projective kernel. Then there is a natural isomorphism $R\otimes_B-\cong\Hom_B(W,-)$ of functors $B-\modu \to R-\modu$ and a natural isomorphism $W\otimes_B-\cong \Hom_B(L,-)$ of functors $B-\modu\to L-\modu$. In particular these functors are exact.
\end{prop}

\begin{thm}[{\cite[Theorems 2.2, 2.4, 2.5]{BB91}}]\label{inducedboxmodules}
Let $\mathfrak{B}=(B,W)$ be a box with projective kernel. The categories $\Ind(B,R)$ and $\CoInd(B,L)$ of (co-)induced modules are equivalent to the category $\mathfrak{B}-\modu$ of representations of the box $\mathfrak{B}$.
\end{thm}

\subsection{Comparing extensions}

For the 'if' part of our main theorem the following result on comparison of extensions turns out to be useful for proving that the induced modules form a standardisable set of modules:

\begin{thm}[{\cite[Corollary 3.5]{BB91}}]\label{extensions}
Suppose $\mathfrak{B}=(B,W)$ is a box with projective kernel. Then the categories $\Ind(B,R)$ and $\CoInd(B,L)$ are closed under extensions and every such extension is induced by an extension in $B-\modu$.
Furthermore there are maps
$$\Ext^n_B(X,Y)\to \Ext^n_R(R\otimes_B X, R\otimes_B Y)$$
and
$$\Ext^n_B(X,Y)\to \Ext^n_L(\Hom_B(L,X),\Hom_B(L,Y))$$
which are epimorphisms for $n=1$ and isomorphisms for $n\geq 2$.
\end{thm}

Furthermore the following theorem will imply Ringel's theorem about the existence of almost split sequences for the category of filtered modules:

\begin{thm}[{\cite[Theorem 4.1]{BB91}\cite[14.3,14.4 Proposition]{Bur05}}]\label{almostsplitbocses}
Suppose that $\mathfrak{B}=(B,W)$ is a box with projective kernel such that $\mathfrak{B}-\modu$ is fully additive. Then the category $\mathfrak{B}-\modu$ (and hence also the equivalent categories $\Ind(B,R)$ and $\CoInd(B,L)$) admit almost split sequences. Furthermore the functors $R\otimes_B-:B-\modu\to R-\modu$ and $\Hom_B(L,-):B-\modu\to L-\modu$ send almost split sequences with indecomposable end terms in $\mathfrak{B}-\modu$ to almost split sequences or to split sequences.
\end{thm}

\subsection{Double centraliser and Ext-injectives}

A certain double centraliser property is ubiquitous in the theory of quasi-hereditary algebras, thus it is mandatory to also have it for the boxes we want to consider:

\begin{thm}[{\cite[Proposition 2.7]{BB91}}]\label{doublecentraliser}
Let $\mathfrak{B}=(B,W)$ be a box with projective kernel. Then there is a double centraliser property:
$$L\cong \End_{R^{op}}(W)\text{ and } R^{op}\cong \End_L(W).$$
\end{thm}

To prove that the left and right algebra of a box are Ringel dual to each other the following proposition will be useful:

\begin{thm}[{\cite[Theorem 5.2]{BB91}}]\label{injectivetiltings}
Suppose that $\mathfrak{B}=(B,W)$ is a box with projective kernel such that $\mathfrak{B}-\modu$ is fully additive. Then a module in $\Ind(B,R)$ is Ext-projective iff it is projective and Ext-injective iff it belongs to $\add(DW)$.
\end{thm}

\section{Proofs of the main results, and some corollaries}\label{section10}

In this section we will collect all the results and prove our main theorems. We start by definining the class of directed boxes:

\begin{defn}
The \emphbf{bigraph} of a box with projective kernel is given by the quiver of $B$ as the set of degree $0$ arrows and for every copy of $Be_\mathtt{i}\otimes e_\mathtt{j}B$ in $W$, a dashed arrow $\begin{xy}\xymatrix{\mathtt{j}\ar@{-->}[r] &\mathtt{i}}\end{xy}$.
A box with projective kernel is called \emphbf{directed} provided its bigraph is directed as an oriented graph.
\end{defn}

The following theorem is the 'if' direction of our main theorem. It appeared in a special case in a slightly different language already in \cite{BH00} (see also \cite{Den01,XL02}).

\begin{thm}
Let $\mathfrak{B}=(B,W)$ be a directed box. Then the left and right algebra of $B$ are quasi-hereditary algebras.
\end{thm}

\begin{proof}
Let $\Delta(\mathtt{i}):=R\otimes_B L(\mathtt{i})$. Then according to the $\Ext$-comparison (Theorem \ref{extensions}) these modules have no extensions in one direction (because the algebra $B$ has none). For the homomorphisms note that $\Hom_R(\Delta(\mathtt{i}),\Delta(\mathtt{j}))\cong \Hom_\mathfrak{B}(L(\mathtt{i}),L(\mathtt{j}))=\Hom_B(W\otimes L(\mathtt{i}),L(\mathtt{j}))$. Since $W$ is a factor of a bimodule with direct summands $Be_\mathtt{k}\otimes e_\mathtt{l} B$ for $\mathtt{k}\geq \mathtt{l}$, $\Hom_B(W\otimes L(\mathtt{i}),L(\mathtt{j}))$ is nonzero only if $\mathtt{i}\leq \mathtt{j}$ as $B$ is directed. Hence, $\Delta(\mathtt{1}),\dots,\Delta(\mathtt{n})$ form a standard system and $R\cong R\otimes_B B$ is a projective generator. Thus $R$ is quasi-hereditary.\\
The proof for $L$ gives modules $\nabla(\mathtt{i}):=\Hom_B(L,L(\mathtt{i}))$ which form a standard system, such that $\mathcal{F}(\nabla)$ contains the injective cogenerator $D(L)\cong \Hom_B(L,DB)$. Hence $L$ is quasi-hereditary by the dual statement of the standardisation theorem.
\end{proof}

The following part is the 'only if' direction which is the direction where all the machinery developed in this article is needed.

\begin{thm}
Let $\Delta$ be a standard system in an abelian $\mathbbm{k}$-category and let $\mathcal{F}(\Delta)$ be the category of objects with a $\Delta$-filtration. Then there exists a directed box $\mathfrak{B}=(B,W)$ such that $\mathcal{F}(\Delta)$ is equivalent as an exact $\mathbbm{k}$-category to the category $\mathfrak{B}-\modu$ of finitely generated $\mathfrak{B}$-representations.

Moreover $\mathcal{F}(\Delta)$ is equivalent as an exact $\mathbbm{k}$-category to the category of $\Delta$-filtered modules for the quasi-hereditary algebra $R_\mathfrak{B}$.

In particular, every quasi-hereditary algebra $(A,\leq)$ is Morita equivalent with the same quasi-hereditary structure to the right Burt-Butler algebra of a directed box.
\end{thm}

\begin{proof}
We have already constructed a directed box $\mathfrak{B}$ with the wanted properties in the previous sections. By the previous theorem, the right Burt-Butler algebra of that box is a quasi-hereditary algebra and the results from the previous section by Burt and Butler guarantee that the category of filtered modules $\mathcal{F}(\Delta)$ is equivalent for the quasi-hereditary algebras $A$ and $R_\mathfrak{B}$. Then the uniqueness-part of the Dlab-Ringel standardisation theorem (Theorem \ref{dlabringelstandardization}) tells us that the two algebras must be Morita equivalent.
\end{proof}

Note that when starting the translation procedure upon which the above proof is based with a quasi-hereditary algebra $(A,\leq)$, the algebra $A$ gets lost already in the first step, which only keeps the category $\mathcal{F}(\Delta)$ up to categorical equivalence. Hence the resulting quasi-hereditary algebra $R_\mathfrak{B}$ is Morita equivalent, but not necessarily isomorphic to $A$.

In the Appendix \ref{A3} we will recall and discuss an example of \cite{Koe95}, which in the present context shows that $R_\mathfrak{B}$ need not, in general, be isomorphic to $A$ itself.

\begin{cor}
For every quasi-hereditary algebra $A$ there is a Morita equivalent algebra $R$ which has an exact Borel subalgebra $B$.
\end{cor}

\begin{proof}
Let $B$ be the algebra from the box $\mathfrak{B}=(B,W)$ constructed above. Obviously $B$ is directed, $B\to R, b\mapsto b\varepsilon$ provides an inclusion, the tensor functor is exact by Proposition \ref{exact} and sends the simples to the standard modules by construction.
\end{proof}

\begin{thm}
Let $\mathfrak{B}$ be a box with projective kernel. Then $L_{\mathfrak{B}}$ is Morita equivalent to the Ringel dual of $R_{\mathfrak{B}}$. In particular $\mathcal{F}(\nabla_{L})\cong \mathcal{F}(\Delta_R)$ and vice versa.
\end{thm}

\begin{proof}
Since $\mathfrak{B}-\modu$ is equivalent to $\mathcal{F}(\Delta)$, it is of course fully additive. Hence (by Theorem \ref{injectivetiltings}) $DW$ contains exactly the Ext-injective objects of $\mathfrak{B}-\modu\cong \mathcal{F}(\Delta)$ as direct summands. It is well known (see e.g. \cite{Rin91}) that the same is true for $T$. Hence $\add DW\cong \add T$ via the restriction of the equivalence $\mathfrak{B}-\modu\cong A-\modu$, where $T$ is the characteristic tilting module of $A$. Thus $\End_R(DW)^{op}\cong \End_{R^{op}}(W)\cong L$ (by Theorem \ref{doublecentraliser}) is Morita equivalent to the Ringel dual of $A$.\\
The last claim follows since by Theorem \ref{inducedboxmodules} both are equivalent to $\mathfrak{B}-\modu$ via inducing (respectively coinducing).
\end{proof}

\begin{cor}[{\cite[Theorem 2]{Rin91}}] \label{existence-of-ASS}
The category $\mathcal{F}(\Delta)$ has almost split sequences. Furthermore whenever $(B,W)$ is a directed box, then the functor $R\otimes_B -: B-\modu B\to \mathcal{F}(\Delta_R)$ sends almost split sequences with indecomposable end terms in $\mathfrak{B}-\modu$ to almost split sequences or to split sequences.
\end{cor}

\begin{proof}
This follows immediately from Theorem \ref{almostsplitbocses}.
\end{proof}

Also other results can now be reproven from the corresponding results for boxes, e.g. the following:

\begin{cor}[{\cite[Theorem 5]{Rin91}}]
The characteristic tilting module is a tilting module.
\end{cor}

\begin{proof}
The corresponding result is that $W$ is a cotilting module by \cite[Theorem 5.1]{BB91}. Thus, $DW$ is a tilting module.
\end{proof}

With repsect to extensions a directed box of a quasi-hereditary algebra contains more information than an arbitrary exact Borel subalgebra:

\begin{cor}
Let $\mathfrak{B}=(B,W)$ be a directed box, $R_\mathfrak{B}$ be its right algebra. Then the class of induced modules from $B$ to $R_{\mathfrak{B}}$ is closed under extensions. Given $B$-modules $X,Y$ there is a map
$$\Ext^n_B(X,Y)\to \Ext^n_R(R\otimes_B X, R\otimes_B Y)$$
which is an epimorphism for $n=1$ and an isomorphism for $n\geq 2$. A similar claim holds for $L_\mathfrak{B}$.
\end{cor}

\begin{proof}
This follows immediately from Theorem \ref{extensions}.
\end{proof}

An example where this is not true for an exact Borel subalgebra not coming from a box is discussed in Appendix \ref{A4}.

The Ringel dual of an algebra can also be constructed as the opposite algebra of the right algebra of the 'opposite' box:

\begin{prop}
Let $\mathfrak{B}=(B,W)$ be a box. Then there is a box $\mathfrak{B}^{op}=(B^{op},W^{op})$ such that $L_\mathfrak{B}\cong R_{\mathfrak{B}^{op}}^{op}$. In particular $R(R(A)^{op})^{op}\cong A$.
\end{prop}

\begin{proof}
Let $\mathfrak{B}=(B,W)$ be a box. Then $B^{op}$ is also a category and each $B$-$B$-bimodule can also be regarded as a $B^{op}$-$B^{op}$-bimodule via the equivalence of right $B$- and left $B^{op}$-modules (and vice versa). Denote $W$, regarded as a $B^{op}$-$B^{op}$-bimodule, as $W^{op}$. Write $\mu$ in Sweedler notation, i.e. $\mu(v)=\sum v_{(1)}\otimes v_{(2)}$,  then  $\mu^{op}$ is given by $\mu^{op}(v)=\sum v_{(2)}\otimes v_{(1)}$. Hence if we compare the multiplication $e\cdot f$ in the left algebra $\Hom_B(W_B,B_B)$:
$$W\stackrel{\mu}{\to} W\otimes_B W\stackrel{f\otimes 1}{\to} B\otimes_B W\stackrel{\sim}{\to} W\stackrel{e}{\to} B$$
with the multiplication $e^{op}*f^{op}=f^{op}\cdot e^{op}$ in the opposite algebra of the right algebra $\Hom_B({}_BW,{}_BB)^{op}$:
$$W^{op}\stackrel{\mu^{op}}{\to} W^{op}\otimes_{B^{op}} W^{op}\stackrel{1\otimes f^{op}}{\to} W^{op}\otimes_{B^{op}} B^{op}\stackrel{\sim}{\to} W^{op}\stackrel{e^{op}}{\to} B^{op}$$
we get the same resulting map $v\mapsto e(f(v_{(1)})v_{(2)})=e^{op}(v_{(2)}*f^{op}(v_{(1)}))$.\\
Since taking the opposite algebra twice gives back the original algebra, there is an isomorphism $R(R(A)^{op})^{op}\cong A$.
\end{proof}

\begin{appendix}

\section{Examples}

\subsection{The regular block of $\mathfrak{sl}_2$}\label{A1}

Our series of examples starts with the basic algebra $A$ corresponding to a regular block of $\mathfrak{sl}_2$, that is the algebra given by the following quiver:
$$\begin{xy}\xymatrix{{\mathtt{2}}\ar@<2pt>[r]^\alpha &1\ar@<2pt>[l]^\beta}\end{xy}$$
with the relation $\beta\alpha$, where we compose arrows like linear maps. It has the following projective modules (written in terms of their socle series) with standard modules having composition factors corresponding to the \framebox{'boxed'} vertices:
$$
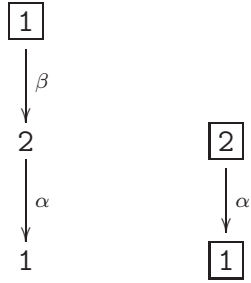
\begin{xy}\xymatrix{ {\framebox{$\mathtt{1}$}}\ar[d]^\beta \\ {\mathtt{2} }\ar[d]^\alpha&&{\framebox{$\mathtt{2}$}}\ar[d]^\alpha \\{\mathtt{1}}&&{\framebox{$\mathtt{1}$}}}\end{xy}$$
The $A_\infty$-Yoneda category $\mathcal{A}$ looks as follows:
$$\begin{xy}\xymatrix{{\Delta(\mathtt{1})}\ar@<2pt>[r]\ar@<-2pt>@{-->}[r]&{\Delta(\mathtt{2})}}\end{xy}$$
where the solid arrow (of degree $1$) corresponds to the extension of $\Delta(\mathtt{2})$ by $\Delta(\mathtt{1})$ given by $P(\mathtt{1})$ and the dashed arrow corresponds to the inclusion of $\Delta(\mathtt{1})$ in $\Delta(\mathtt{2})$. The multiplications all equal zero except for the ones with the identity morphisms that we have omitted in our picture. This is because $b_1$ is always zero since the Yoneda category is a minimal model, i.e. the homology of some algebra, and $b_{\geq 2}$ are zero as there are no paths of length $\geq 2$ in the graded quiver. Now we have to compute $\mathbb{D}s\mathcal{A}$ according to our procedure. This is given by the following quiver (here we have included the duals of the identity morphisms).
$$\begin{xy}\xymatrix{\Delta(\mathtt{1})\ar@{-->}@(l,dl)_{\omega_1}\ar@<2pt>[r]^{a}\ar@{-->}@<-2pt>[r]_{\varphi}&\Delta(\mathtt{2})\ar@{-->}@(r,dr)^{\omega_\mathtt{2}}}\end{xy}$$
So a Borel subalgebra is given by the path algebra of the Dynkin quiver $\mathbb{A}_\mathtt{2}$. To compute the right algebra of the box (which according to our  results is Morita equivalent to the regular block of $\mathfrak{sl}_2$) we have to compute the opposite of the endomorphism algebra of $B$, as a representation of the box. \\
A representation of that box is by definition a representation of $B$ and a morphism of box representations is an assignment of linear maps making the following square commutative (were the map on the diagonal doesn't satisfy any conditions since $\partial a=0$):
$$\begin{xy}\xymatrix{V_\mathtt{1}\ar@{-->}[r]^{f_{\omega_\mathtt{1}}}\ar[d]_{V_a}\ar@{-->}[rd]|{f_{\varphi}} &W_\mathtt{1}\ar[d]^{W_a}\\V_\mathtt{2}\ar@{-->}[r]_{f_{\omega_\mathtt{2}}} &W_\mathtt{2} }\end{xy}$$
So for the box representation $B$ we get the following diagram:
$$\begin{xy}\xymatrix{\mathbbm{k}\ar@{-->}[r]^{x}\ar[d]_{\begin{pmatrix}1\\0\end{pmatrix}}\ar@{-->}[rd]|{\begin{pmatrix}\lambda\\\mu\end{pmatrix}} &\mathbbm{k}\ar[d]^{\begin{pmatrix}1\\0\end{pmatrix}}\\\mathbbm{k}^2\ar@{-->}[r]_{\begin{pmatrix}x&y\\0&z\end{pmatrix}} &\mathbbm{k}^2 }\end{xy}$$
The resulting opposite endomorphism algebra is $5$-dimensional, and isomorphic to $A$. To check this fact by computing the composition of such maps, we write two such morphisms next to each other (one with primes and one without). Since $\partial v=0$ we just have to compute the composition corresponding to $\omega_\mathtt{2} v+v\omega_\mathtt{1}$ to get the resulting map on the diagonal. Since the $\omega_\mathtt{i}$ are grouplike, the corresponding maps are given by the composition of the maps corresponding to the $\omega_\mathtt{i}$:
$$\begin{xy}\xymatrix@C=2cm{\mathbbm{k}\ar@{-->}[r]^{x}\ar[dd]_{\begin{pmatrix}1\\0\end{pmatrix}}\ar@{-->}[rdd]|{\begin{pmatrix}\lambda\\\mu\end{pmatrix}} &\mathbbm{k}\ar@{-->}[rdd]|{\begin{pmatrix}\lambda'\\\mu'\end{pmatrix}}\ar[dd]|{\begin{pmatrix}1\\0\end{pmatrix}}\ar@{-->}[r]^{x'}&\mathbbm{k}\ar[dd]^{\begin{pmatrix}1\\0\end{pmatrix}}&&\mathbbm{k}\ar@{-->}[r]^{x'x}\ar@{-->}[rdd]|{\begin{pmatrix}x'\lambda+y'\mu+\lambda'x\\z'\mu+\mu' x\end{pmatrix}}&\mathbbm{k}\\&&&=\\\mathbbm{k}^2\ar@{-->}[r]_{\begin{pmatrix}x&y\\0&z\end{pmatrix}} &\mathbbm{k}^2\ar@{-->}_{\begin{pmatrix}x'&y'\\0&z'\end{pmatrix}}[r]&\mathbbm{k}^2&&\mathbbm{k}^2\ar@{-->}[r]_{\begin{pmatrix}x'x&x'y+y'z\\0&z'z\end{pmatrix}}&\mathbbm{k}^2}\end{xy}$$
We have omitted the solid arrows in the rightmost diagram to make it more readable. This algebra can be represented by the following matrix algebra which is isomorphic to $A^{op}$:
$$\begin{pmatrix}x&y&\lambda\\0&z&\mu\\0&0&x\end{pmatrix}$$
That the right algebra of the box is isomorphic (and not only Morita equivalent) to the original basic algebra we started with seems to be a low rank phenomenon. Even in the case of blocks of category $\mathcal{O}$ (which were proven to have exact Borel subalgebras in \cite{Koe95}) we will usually get 'bigger' algebras. This happens for example already for the singular block of $\mathfrak{sl}_3$ or for the block of the parabolic version of category $\mathcal{O}$ corresponding to the Levi subalgebra $\mathfrak{gl}_1(\mathbb{C})\oplus \mathfrak{gl}_2(\mathbb{C})\subseteq \mathfrak{gl}_3(\mathbb{C})$. In the latter case the $A_\infty$-structure on the Ext-algebra was computed by Klamt and Stroppel in \cite{KS12}.

\subsection{An example from the reduction algorithm}\label{A2}

The next example gives counterexamples to some tempting conjectures, i.e. that the directed box giving a quasi-hereditary algebra is unique or that the connectedness of the biquiver given by the box implies the connectedness of the algebra. We modify the box obtained in the last example slightly, so that $a$ becomes a so-called irregular or superfluous arrow:
$$\begin{xy}\xymatrix{\mathtt{1}\ar@{-->}@(l,dl)_{\omega_\mathtt{1}}\ar@<2pt>[r]^a\ar@{-->}@<-2pt>[r]_v&\mathtt{2}\ar@{-->}@(r,dr)^{\omega_\mathtt{2}}}\end{xy}$$
with $\partial a =v$ (and $\partial v=0$). A box representation is still a representation of $B$ as before but morphisms involve the diagonal, i.e. by (d2) we must have $f_{\omega_\mathtt{2}}V_a-W_af_{\omega_\mathtt{1}}+f_\varphi=0$.
$$\begin{xy}\xymatrix{V_\mathtt{1}\ar@{-->}[r]^{f_{\omega_\mathtt{1}}}\ar[d]_{V_a}\ar@{-->}[rd]|{f_{\varphi}} &W_\mathtt{1}\ar[d]^{W_a}\\V_\mathtt{2}\ar@{-->}[r]_{f_{\omega_\mathtt{2}}} &W_\mathtt{2} }\end{xy}$$
For the possible endomorphisms of $B$ we thus get:
$$\begin{xy}\xymatrix@!=1.5cm{\mathbbm{k}\ar@{-->}[r]^\lambda\ar@{-->}[rd]|{\begin{pmatrix}\lambda-r\\-t\end{pmatrix}}&\mathbbm{k}\\\mathbbm{k}^2\ar@{-->}[r]_{\begin{pmatrix}r&s\\t&u\end{pmatrix}}&\mathbbm{k}^2}\end{xy}$$
It is easy to see (the diagonal is determined by the matrix rings) that the composition of two morphisms just corresponds to the multiplication in the matrix rings. Thus the right algebra (which is again $5$-dimensional) of the box is:
$$\begin{pmatrix}\lambda&0&0\\0&r&s\\0&t&u\end{pmatrix}$$
If we would do our process with the basic algebra corresponding to that algebra, i.e. $\mathbbm{k}\times \mathbbm{k}$ we would end up with the following box:
$$\begin{xy}\xymatrix{\mathtt{1}\ar@{-->}@(l,dl)_{\omega_\mathtt{1}}&\mathtt{2}\ar@{-->}@(r,dr)^{\omega_\mathtt{2}}}\end{xy}$$
That the two boxes have equivalent representation categories was known before under the name of 'regularisation'.

\subsection{An algebra not having an exact Borel subalgebra}\label{A3}

As remarked above, not every algebra has an exact Borel subalgebra. The following example mentioned in \cite{Koe95} is such an algebra. But there is an error in the argument in \cite{Koe95}. In contrast to our main theorem it is claimed that there is no Morita equivalent version having an exact Borel subalgebra. Here we will compute a Morita equivalent algebra having an exact Borel subalgebra. This also contradicts the result in \cite{Koe99} that having an exact Borel subalgebra is Morita invariant. The algebra is given by quiver and relations as follows:
 $$
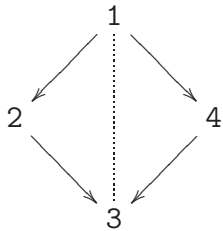
\begin{xy}\xymatrix{&\mathtt{1}\ar[rd]\ar[ld]\ar@{.}[dd]\\\mathtt{2}\ar[rd]&&\mathtt{4}\ar[ld]\\&\mathtt{3}}\end{xy}$$
where the dotted line in the middle indicates the obvious commutativity relation. We get the following projective modules. Again the composition factors of the corresponding standard modules have been boxed.
$$
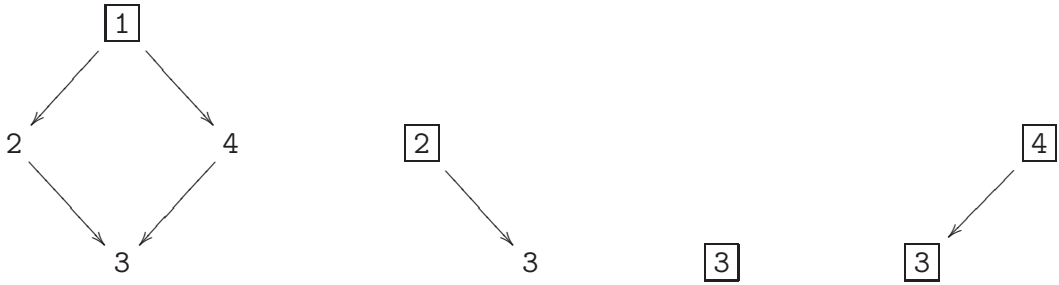
\begin{xy}\xymatrix{
&\framebox{$\mathtt{1}$}\ar[rd]\ar[ld]\\\mathtt{2}\ar[rd]&&\mathtt{4}\ar[ld]&&\framebox{$\mathtt{2}$}\ar[rd]&&&&&&\framebox{$\mathtt{4}$}\ar[ld]\\&\mathtt{3}&&&&\mathtt{3}&&\framebox{$\mathtt{3}$}&&\framebox{$\mathtt{3}$}
}\end{xy}$$
The following diagram gives the $A_\infty$-Yoneda category:
$$
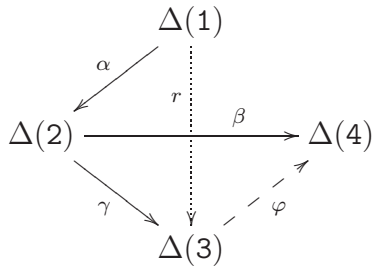
\begin{xy}\xymatrix{&\Delta(\mathtt{1})\ar[ld]_\alpha\ar@{.>}[dd]_(.33)r\\\Delta(\mathtt{2})\ar[rr]^(.66)\beta\ar[rd]_{\gamma}&&\Delta(\mathtt{4})\\&\Delta(\mathtt{3})\ar@{-->}[ru]_\varphi}\end{xy}$$
Since most of the standard modules are simple, the arrows of degree $1$ on the left and the dotted arrow of degree $2$ correspond to arrows in the original quiver. The solid arrow $\begin{xy}\xymatrix{\Delta(\mathtt{2})\ar[r] &\Delta(\mathtt{4})}\end{xy}$ corresponds to the extension given by $\operatorname{rad} P(\mathtt{1})$ and the dashed arrow $\begin{xy}\xymatrix{\Delta(\mathtt{3})\ar@{-->}[r] &\Delta(\mathtt{4})}\end{xy}$ corresponds to the inclusion.\\
All the $b_i$ are zero for $i\geq 4$, since there are no paths of length $\geq 4$. There is only one path of length $3$, but there is no arrow (of degree $1$) from $\Delta(\mathtt{1})$ to $\Delta(\mathtt{4})$, hence $b_3$ is also zero. Since $b_2$ is induced by the Yoneda product, it can be calculated quite easily by choosing representatives of the arrows. We omit the calculations here.\\
This results in the differentials given by $d(r)=\gamma\alpha$, $\partial \beta=\varphi\gamma$ and zero in the other cases. Thus a Borel subalgebra $B$ is given by the solid arrows in this quiver and the relation $d(r)=0$.\\
Now $B$ is given by the following representation where the arrows are numbered from top to bottom:
$$\begin{xy}\xymatrix{\mathbbm{k}\ar[d]_{\begin{pmatrix}1\\0\end{pmatrix}}\\\mathbbm{k}^2\ar[d]^{\begin{pmatrix}0&1\\0&0\end{pmatrix}}\ar@/_0.5cm/[dd]_{\begin{pmatrix}1&0\\0&1\\0&0\end{pmatrix}}\\\mathbbm{k}^2\\\mathbbm{k}^3
}\end{xy}$$
For the morphisms we will again omit the solid arrows to save space. The commutativities give the following endomorphisms:
$$\begin{xy}\xymatrix@!=1.5cm{
\mathbbm{k}\ar@{-->}[r]^a&\mathbbm{k}\\
\mathbbm{k}^2\ar@{-->}[r]^{\begin{pmatrix}a&c\\0&d\end{pmatrix}}&\mathbbm{k}^2\\
\mathbbm{k}^2\ar@{-->}[r]^{\begin{pmatrix}d&b\\0&e\end{pmatrix}}\ar@{-->}[rd]|{\begin{pmatrix}f-c&\ell\\h-d&m\\j&n\end{pmatrix}}&\mathbbm{k}^2\\
\mathbbm{k}^3\ar@{-->}[r]_{\begin{pmatrix}a&f&g\\0&h&i\\0&j&k\end{pmatrix}}&\mathbbm{k}^3
}\end{xy}$$
The composition is as follows:
$$\begin{xy}\xymatrix@!=1.5cm{
\mathbbm{k}\ar@{-->}[r]^a&\mathbbm{k}\ar@{-->}[r]^{a'}&\mathbbm{k}&&\mathbbm{k}\ar@{-->}[rr]^{a'a}&&\mathbbm{k}\\
\mathbbm{k}^2\ar@{-->}[r]^{\begin{pmatrix}a&c\\0&d\end{pmatrix}}&\mathbbm{k}^2\ar@{-->}[r]^{\begin{pmatrix}a'&c'\\0&d'\end{pmatrix}}&\mathbbm{k}^2&&\mathbbm{k}^2\ar@{-->}[rr]^{\begin{pmatrix}a'a&a'c+c'd\\0&d'd\end{pmatrix}}&&\mathbbm{k}^2\\
\mathbbm{k}^2\ar@{-->}[r]^{\begin{pmatrix}d&b\\0&e\end{pmatrix}}\ar@{-->}[rd]|{\begin{pmatrix}f-c&\ell\\h-d&m\\j&n\end{pmatrix}}&\mathbbm{k}^2\ar@{-->}[r]^{\begin{pmatrix}d'&b'\\0&e'\end{pmatrix}}\ar@{-->}[rd]|{\begin{pmatrix}f'-c'&\ell'\\h'-d'&m'\\j'&n'\end{pmatrix}}&\mathbbm{k}^2&=&\mathbbm{k}^2\ar@{-->}[rr]^{\begin{pmatrix}d'&b'\\0&e'\end{pmatrix}}\ar@{-->}[rrd]|{\tiny{\begin{pmatrix}-c'd+a'f-a'c+f'h+gj&f'b-c'b+\ell'e+a'\ell+f'm+g'n\\-d'd+h'h+i'j&h'b-d'b+m'e+h'm+i'n\\j'h+k'j&j'b+n'e+j'm+k'n\end{pmatrix}}}&&\mathbbm{k}^2\\
\mathbbm{k}^3\ar@{-->}[r]_{\begin{pmatrix}a&f&g\\0&h&i\\0&j&k\end{pmatrix}}&\mathbbm{k}^3\ar@{-->}[r]_{\begin{pmatrix}a'&f'&g'\\0&h'&i'\\0&j'&k'\end{pmatrix}}&\mathbbm{k}^3&&\mathbbm{k}^3\ar@{-->}[rr]_{\tiny{\begin{pmatrix}a'a &a'f+f'h+g'j &a'g+f'i+g'h\\0&h'h+i'j&h'i+i'k\\0&j'h+k'j&j'i+k'k \end{pmatrix}}}&&\mathbbm{k}^3
}\end{xy}$$
One checks quite straightforwardly that setting one of the parameters $a$, $d$, $e$, $h$ or $k$ to $1$ and all others to $0$ one gets an idempotent and that $a\ell e=\ell$, $hme=m$ and $kne=n$, where the latter corresponds to the parameter set to $1$ and all others to $0$. Explicit calculations show that this gives the following algebra which is Morita equivalent, but not isomorphic, to the algebra $A^{op}$:
$$\begin{pmatrix}a&c&f&g&\ell\\
0&d&0&0&-b\\0&0&h&i&m+b\\
0&0&j&k&n\\0&0&0&0&e\end{pmatrix}.$$
We can quite easily see the Borel subalgebra $B$. If $a\in B$ then $(a\cdot \varepsilon)(\varphi)=\varepsilon(\varphi\cdot a)=\varepsilon(\varphi)\cdot a=0$ by definition and since $\varepsilon$ is a $B$-$B$-bimodule homomorphism. This implies that $f=c$, $h=d$ and $j=\ell=m=n=0$. The remaining parameters form an $8$-dimensional algebra, which needs to be isomorphic to $B^{op}$ by the theory we developed. Thus, the opposite of the Borel subalgebra $B^{op}$ looks as follows:
$$\begin{pmatrix}a&c&c&g&0\\0&d&0&0&-b\\0&0&d&i&b\\0&0&0&k&0\\0&0&0&0&e\end{pmatrix}.$$
Alternatively one can compute by hand that $ae_1\varepsilon+de_2\varepsilon+ee_3\varepsilon+ke_4\varepsilon+c\alpha\varepsilon+b\gamma\varepsilon+i\beta\varepsilon+g\beta\alpha\varepsilon$ is a subalgebra isomorphic to $B^{op}$.

\subsection{An exact Borel subalgebra not associated with a directed box}\label{A4}

Note that the exact Borel subalgebras we give do not exhaust all exact Borel subalgebras. The exact Borel subalgebras given by our process have the special property that the induction functor is dense on the category of $\Delta$-filtered modules (see Theorem \ref{extensions}). An exact Borel subalgebra where this property fails was given in \cite{Koe99}. It is the algebra given by quiver:
$$
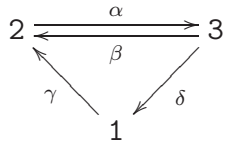
\begin{xy}\xymatrix{
\mathtt{2}\ar@<2pt>[rr]^\alpha&&\mathtt{3}\ar@<2pt>[ll]^\beta\ar[dl]^\delta\\
&\mathtt{1}\ar[ul]^\gamma
}
\end{xy}$$
with relations $\delta\gamma=0=\alpha\beta$. According to \cite{Koe99} it has a Borel subalgebra given by the subquiver:
$$
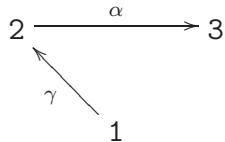
\begin{xy}\xymatrix{
\mathtt{2}\ar@<2pt>[rr]^\alpha&&\mathtt{3}\\
&\mathtt{1}\ar[ul]^\gamma
}
\end{xy}$$
We use our algorithm to produce an exact Borel subalgebra (of a Morita equivalent algebra) such that the induction functor is dense on $\mathcal{F}(\Delta)$. Again we start by writing down the projective modules (with boxed standard modules):
$$
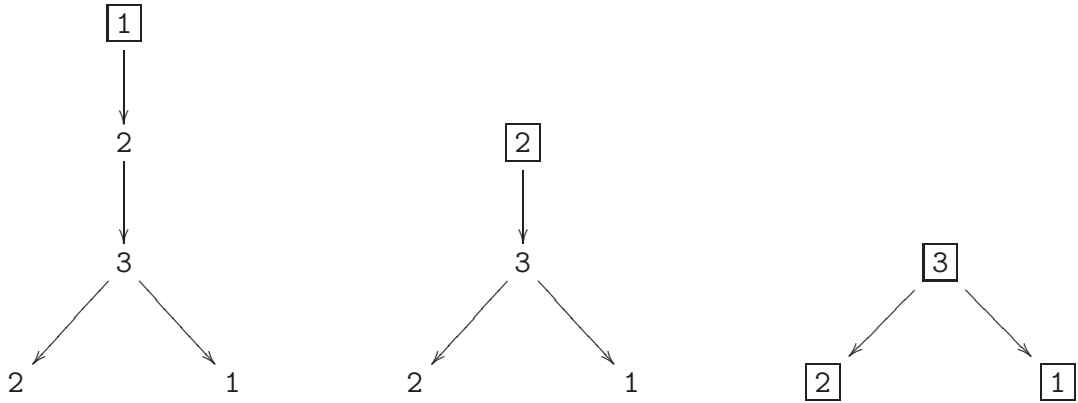
\begin{xy}\xymatrix{
&\framebox{$\mathtt{1}$}\ar[d]\\
&\mathtt{2}\ar[d]&&&&\framebox{$\mathtt{2}$}\ar[d]\\
&\mathtt{3}\ar[rd]\ar[ld]&&&&\mathtt{3}\ar[ld]\ar[rd]&&&&\framebox{$\mathtt{3}$}\ar[rd]\ar[ld]\\
\mathtt{2}&&\mathtt{1}&&\mathtt{2}&&\mathtt{1}&&\framebox{$\mathtt{2}$}&&\framebox{$\mathtt{1}$}
}
\end{xy}$$
Then the Ext-category is given by the following graded quiver (it is easy to see that the $\Delta$ have a projective resolution of length $2$, hence $\Ext^2$ vanishes):
$$
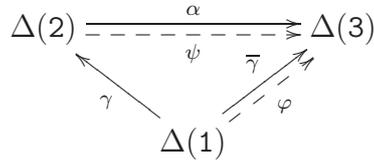
\begin{xy}\xymatrix{
\Delta(\mathtt{2})\ar@<2pt>[rr]^\alpha\ar@{-->}@<-2pt>[rr]_\psi&&\Delta(\mathtt{3})\\
&\Delta(\mathtt{1})\ar@<2pt>[ur]^{\overline{\gamma}}\ar@{-->}@<-2pt>[ur]_{\varphi}\ar[ul]^{\gamma}
}
\end{xy}$$
The following pushout diagram shows that $b_2(\psi,\gamma)=\overline{\gamma}$ and the reason why we named $\overline{\gamma}$ like that:
$$
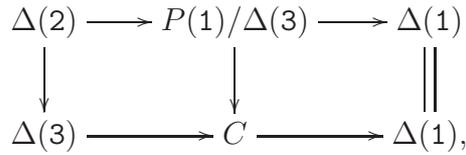
\begin{xy}\xymatrix{
\Delta(\mathtt{2})\ar[r]\ar[d]&P(\mathtt{1})/\Delta(\mathtt{3})\ar[r]\ar[d]&\Delta(\mathtt{1})\ar@<2pt>@{-}[d]\ar@<-2pt>@{-}[d]\\
\Delta(\mathtt{3})\ar[r]&C\ar[r]&\Delta(\mathtt{1}),\\
}
\end{xy}$$
where $C$ is the module
$$
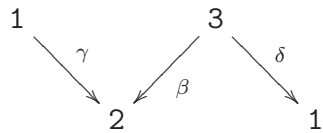
\begin{xy}
\xymatrix{
\mathtt{1}\ar[rd]^\gamma&&\mathtt{3}\ar[rd]^\delta\ar[ld]^\beta\\
&\mathtt{2}&&\mathtt{1}
}
\end{xy}$$
Note that $C$ also is the module which is not reached by the induction functor of the exact Borel subalgebra given above.
It is also easy to see that $b_2(\alpha,\gamma)=0$ and all other $b_i$ are zero by similar reasoning as in the other examples.
Now $B$ is given by the following representation:
$$
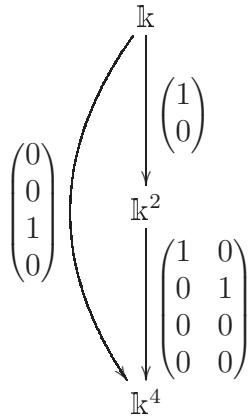
\begin{xy}\xymatrix@R=2cm{
\mathbbm{k}\ar[d]^{\begin{pmatrix}1\\0\end{pmatrix}}\ar@/_1cm/[dd]_{\begin{pmatrix}0\\0\\1\\0\end{pmatrix}}\\
\mathbbm{k}^2\ar[d]^{\begin{pmatrix}1&0\\0&1\\0&0\\0&0\end{pmatrix}}\\
\mathbbm{k}^4
}
\end{xy}$$
When writing the morphisms we again omit the solid arrows:
$$\begin{xy}\xymatrix@!=2cm{
\mathbbm{k}\ar@{-->}[rrr]^{a}\ar@{-->}[rrrdd]|(.2){\begin{pmatrix}r\\s\\t\\u\end{pmatrix}}&&&\mathbbm{k}\\
\mathbbm{k}^2\ar@{-->}[rrrd]|(.4){\begin{pmatrix}b&\ell\\f&m\\h-a&n\\j&o\end{pmatrix}}\ar@{-->}[rrr]^(.7){\begin{pmatrix}a&c\\0&d\end{pmatrix}}&&&\mathbbm{k}^2\\
\mathbbm{k}^4\ar@{-->}[rrr]_{\begin{pmatrix}a&c&b&e\\0&d&f&g\\0&0&h&i\\0&0&j&k\end{pmatrix}}&&&\mathbbm{k}^4
}
\end{xy}$$
We leave it to the reader to check that the following matrix algebra gives an algebra  isomorphic to $R^{op}$ and therefore Morita equivalent to $A^{op}$ and that the given subalgebra gives an exact Borel subalgebra. The methods to check this are the same as in the last example.
$$\begin{pmatrix}a&c&b&e&r&\ell\\0&d&f&g&s&m\\0&0&h&i&t&n+c\\0&0&j&k&u&o\\0&0&0&0&a&0\\0&0&0&0&0&d\end{pmatrix}\supset \begin{pmatrix}a&c&0&e&0&0\\0&d&0&g&0&0\\0&0&a&i&0&c\\0&0&0&k&0&0\\0&0&0&0&a&0\\0&0&0&0&0&d\end{pmatrix}.$$

\end{appendix}

\section*{Acknowledgement}

The authors would like to thank the anonymous referees for detailed comments that improved the readability of this article.

\bibliographystyle{abbrv}
\bibliography{publication}

\printindex

\end{document}